\documentclass[11pt, reqno]{amsart}
\usepackage{amsmath, amsthm, amscd, amsfonts, amssymb, graphicx, color, stmaryrd}
\usepackage{mathrsfs}
\usepackage{float} 
\usepackage[all,cmtip]{xy}	
\usepackage{tikz-cd}
\usepackage[bookmarksnumbered, colorlinks, plainpages]{hyperref}

\tikzset{%
    symbol/.style={%
        ,draw=none
        ,every to/.append style={%
            edge node={node [sloped, allow upside down, auto=false]{$#1$}}}
    }
}

\newcommand{\Nn}{\mathbb{N}}
\newcommand{\Rr}{\mathbb{R}}
\newcommand{\Cc}{\mathbb{C}}
\newcommand{\Zz}{\mathbb{Z}}

\newcommand{\Aa}{\mathbf{A}}	
\newcommand{\D}{\mathcal{D}}	
\newcommand{\F}{\mathcal{F}}
\newcommand{\K}{\mathcal{K}}	
\renewcommand{\H}{\mathcal{H}}	
\renewcommand{\L}{\mathcal{L}}	
\newcommand{\C}{\mathcal{C}}	
\newcommand{\E}{\mathcal{E}}	
\newcommand{\Emb}{\mathscr{E}}	
\newcommand{\EmbV}{\Emb_{\V}} 
\newcommand{\A}{\mathcal{A}} 	
\newcommand{\U}{\mathcal{U}}  	
	
\newcommand{\Ab}{\mathrm{Ab}}	

\newcommand{\fop}{f^{(0)}}		
\newcommand{\varphiop}{\varphi^{(0)}}	
\newcommand{\Tau}{\mathcal{T}}
\newcommand{\id}{\mathrm{id}}

\newcommand{\B}{\mathcal{B}}

\newcommand{\Cinfty}{\mathbf{C^{\infty}}} 
\newcommand{\Abdy}{\widetilde{\A}}	
\newcommand{\Gbdy}{\widetilde{\Gad}} 
\newcommand{\Gr}{\mathrm{Gr}}	
\newcommand{\LA}{\mathcal{LA}} 
\newcommand{\LG}{\mathcal{LG}} 
\renewcommand{\H}{\mathcal{H}} 
\newcommand{\G}{\mathcal{G}}	
\newcommand{\Gad}{\mathrm{\G^{ad}}}	
\newcommand{\CG}{\C \G} 	
\newcommand{\SG}{\mathcal{SG}} 
\newcommand{\Gop}{\mathrm{\G^{(0)}}} 
\newcommand{\Gmor}{\mathrm{\G^{(1)}}} 
\newcommand{\Gpull}{\mathrm{\G^{(2)}}}	

\newcommand{\Hop}{\mathrm{\H^{(0)}}} 
 
\newcommand{\Hpull}{\mathrm{\H^{(2)}}}
\newcommand{\Hol}{\mathrm{Hol}} 
\newcommand{\IG}{\mathbb{I}\G} 
\renewcommand{\NG}{\mathbb{N}\G} 
\newcommand{\Kop}{\mathrm{\K^{(0)}}}

\newcommand{\M}{\mathcal{M}}
\renewcommand{\O}{\mathcal{O}}

\renewcommand{\L}{\mathcal{L}} 

\newcommand{\Set}{\mathbf{Set}}

\newcommand{\Zop}{Z^t}
	% ...

\newcommand{\TM}{\mathrm{TM}}
\newcommand{\V}{\mathcal{V}}		
\newcommand{\W}{\mathcal{W}}		
\newcommand{\N}{\mathcal{N}}		
\newcommand{\End}{\mathrm{End}}		
\newcommand{\Mor}{\mathrm{Mor}}
\renewcommand{\P}{\mathcal{P}}		
\newcommand{\supp}{\mathrm{supp}}	
\newcommand{\CT}{\mathcal{CT}}		
\newcommand{\BG}{\mathrm{\mathbb{B}\G}}	
\newcommand{\lift}{\mathrm{lift}}

\newcommand{\flip}{\operatorname{f}}
\newcommand{\scal}[2]{\langle #1, #2 \rangle}	
\newcommand{\im}{\operatorname{im}} 
\newcommand{\codim}{\operatorname{codim}}
	
\newcommand{\op}{\operatorname{op}} 		
\newcommand{\ind}{\operatorname{ind}}

\newcommand{\simM}{\mathrm{\sim_{\M}}}
\newcommand{\Exp}{\mathrm{Exp}} 	

\newcommand{\iso}{\xrightarrow{\sim}}	

\newtheorem{Thm}{Theorem}[section]
\newtheorem{Lem}[Thm]{Lemma}
\newtheorem{Prop}[Thm]{Proposition}
\newtheorem{Cor}[Thm]{Corollary}
\theoremstyle{definition}
\newtheorem{Def}[Thm]{Definition}
\newtheorem{Exa}[Thm]{Example}

\newtheorem{Rem}[Thm]{Remark}

\newtheorem{Not}[Thm]{Notation}

\begin{document}
\setcounter{page}{1}


\title{Groupoids and singular manifolds}

\author[Karsten Bohlen]{Karsten Bohlen}

\address{$^{1}$ Leibniz University Hannover, Germany}
\email{\textcolor[rgb]{0.00,0.00,0.84}{bohlen.karsten@math.uni-hannover.de}}





\begin{abstract}
We describe how Lie groupoids are used in singular analysis, index theory and non-commutative geometry and give a brief overview of the theory.
We also expose groupoid proofs of the Atiyah-Singer index theorem and discuss the Baum-Connes conjecture for Lie groupoids.  
With the help of the general framework of Lie groupoids and related structures we survey recent progress on problems which were outside the scope of the original work of Atiyah and Singer.
This includes the Atiyah-Singer type index problem for many classes of non-compact manifolds (e.g. manifolds with a Lie structure at infinity). We also consider generalizations of the pseudodifferential calculus on Lie groupoids, e.g. for boundary value problems.
\end{abstract} \maketitle

\section{Introduction}

The Atiyah-Singer index theorem is a fundamental discovery in the history of mathematics. We refer to the main references
\cite{AS}, \cite{ASI}, \cite{ASIII}. The theorem - as originally stated - connects an important invariant in functional analysis, the Fredholm index, 
with topology and geometry. The Atiyah-Singer index formula expresses the Fredholm index of Fredholm operators, defined over a compact (closed) manifold in terms
of topological invariants of the manifold (the Chern character and the Todd class). By now this procedure of expressing an important analytical invariant in terms 
of topological or geometrical information of a manifold (or singular space) has become a principle of mathematics. The focus has shifted from the study of the Atiyah-Singer index theorem, and its many
ramifications, to research on potential generalizations of the result. 
We do not attempt a summary of the vast literature on index theory. Instead we focus on a single viewpoint of index theory, based on Lie groupoid techniques and survey (a very small) portion of the results in the area. 
The main point we want to make in this note is that by properly introducing the category of Lie groupoids and related categories, which are \emph{internal} e.g. to the category of smooth manifolds, we can study generalizations of the index theory 
in terms of a very convenient and flexible framework.

Let us first give a brief summary of the literature on the groupoid approach to the Atiyah-Singer index theorem. 
We apologize for the numerous omissions of important work in the area. In \cite{C} Connes gives a short and elegant proof of the Atiyah-Singer index theorem for closed manifolds by using the \emph{tangent groupoid} which has inspired many further results in the subject.
An exposition of the tangent groupoid proof of the index theorem for Dirac operators can be found in \cite{H}. 
Furthermore, we have been inspired by a survey due to Debord and Lescure about groupoids and index theory, \cite{DL}. 
In \cite{DLN}, section 6.1, the reader can find a version of the proof of the index theorem which has the advantage of being easily generalized to the case of manifolds with singularities.
Some aspects of the groupoid proof are exposed in \cite{L} with a focus on the connections to deformation quantization and the Baum-Connes conjecture. 
The article \cite{CLM} contains an extension of Connes'es proof to the case of a compact manifold with boundary which gives a type of Atiyah-Patodi-Singer index theorem. 

\subsection*{Overview} 

The paper is organized as follows.
In the second Section we review the general framework for the study of the index theory on a class of smooth manifolds which model many types of singular manifolds.
To this end we first introduce categories, first of all the category of Lie groupoids as a category internal to the category of smooth functions on manifolds with corners.
We recall the definition of the category of integrable Lie algebroids, the category of symplectic groupoids and the bimodule category of $C^{\ast}$-algebras. 
As a generalization of the Muhly-Renault-Williams theorem it can be shown that there are functorial relationships between these categories, which at a later stage can be used in the study of the analytic and topological index
of operators defined on a Lie groupoid. In the third Section we recall the definition of the calculus of pseudodifferential operators on Lie groupoids.
We then define in the fourth Section the manifolds with a Lie structure at infinity and discuss examples. We consider two cases of Lie manifolds in some detail.
At first the base case where the manifold is simply a compact manifold without boundary. In this case we describe the structure of Connes' groupoid
proof of the Atiyah-Singer index theorem in some detail in the Sections five through seven. In the eighth Section we recall the Baum-Connes conjecture for Lie groupoids and discuss the recent progress on this conjecture as well as its 
relation to the index theory on singular manifolds. 
Then we examine the index theory with values in the $K$-theory group of the $C^{\ast}$-algebra of a Lie groupoid. 
Here we first specialize to the case of the maximal Lie structure, i.e. the Lie structure of all vector fields tangent to the boundary of a given compact manifold with corners in Section nine.
We note that the generalized analytic index which we defined via the tangent groupoid deformation no longer equals the Fredholm index when the manifold $M$ has a boundary (or more generally corners).
In such a case we analyze the generalized index problem stated as a topological calculation of the analytic index. In the general case the analytic index is defined via an adiabatic deformation groupoid, \cite{MP}. 
We briefly sketch the topological index theorem for manifolds with corners which is due to Monthubert and Nistor, \cite{MN2}. 
Then we consider the setting for boundary value problems on manifolds with corners. We prove a generalization of the topological index theorem for so-called Lie manifolds with boundary. 
At the end of this work, in Section ten, we use the categorical framework as developed in the second Section to introduce a calculus
for relative elliptic problems which is a microlocalization of an embedding of Lie manifolds.

\section{Categories}

We review the definition of a \emph{semi-groupoid} in purely algebraic terms. In this sense a semi-groupoid has the structural maps of a groupoid, except for the operation of taking the inverse,
which is not defined in a semi-groupoid. Our definition of a semi-groupoid is identical to the general notion of internal category with ambient category being the category $\Set$.

\textbf{Ambient category:} Let $\Aa$ be a category. We define the class of \emph{strict epimorphisms} contained in the class of all arrows $\Mor(\Aa)$. This class has the property that the pullback $A \ast B := \{(a, b) \in A \times B : f(a) = g(b)\}$ is an object of $\Aa$ for any two strict epimorphisms $f \colon A \to C, \ g \colon B \to C$. An arrow is called a \emph{strict monomorphism} if it is a 
strict epimorphism in the opposite category $\Aa^{op}$. 

\textbf{Internal category:} We denote by $\C = (\C_0, \C_1)$ the category \emph{internal} to $\Aa$ where $\C_0 \in \Aa$ is the object of objects, $\C_1 \in \Aa$ is the object of arrows.
Additionally, $\C$ is endowed with the structural maps $s \colon \C_1 \to C_0$ and $r \colon \C_1 \to \C_0$ which are called source and range map, respectively.
Fix the unit inclusion $u \colon \C_0 \to \C_1$. We assume that $r$ and $s$ are strict epimorphisms in the ambient category. Furthermore, we set $\C_2 := \C_1 \times_{\C_0} \C_1$ for the pullback along the source and range maps $s \colon C_1 \to C_0$ and $r \colon C_1 \to C_0$. 
Denote by $m \colon \C_2 \to \C_1$ the multiplication map. We can summarize the structural maps in terms of the sequence
\[
\xymatrix{
\C_2 \ar@{->>}[r]^{m} & \C_1 \ar@{->>}@< 2pt>[r]^-{r,s} \ar@{->>}@<-2pt>[r]^{} & \C_0 \ar@{>->}[r]^{u} & C_1.
}
\]

We impose the following axioms on the structural maps:
\emph{(i)} $(s \circ u)_{|\C_0} = (r \circ u)_{|\C_0} = \id_{\C_0}$.

\emph{(ii)} For each $a \in \C_1$
\[
m((u \circ r)(a), a) = a, \ m(a, (u \circ s)(a)) = a.
\]

\emph{iii)} For $(a, b) \in \C_2$ we have
\[
r(m(a,b)) = r(a), \ s(m(a,b)) = s(b).
\]

\emph{(iv)} For $(a_1, a_2), \ (b_2, b_3) \in \C_2$ we have
\[
m(m(a_1, a_2), a_3) = m(a_1, m(a_2, a_3)).
\]

\textbf{Internal groupoid:} Most of the time we assume additionally that $\C = (\C_0, \C_1)$ is a groupoid. This means there exists an operation of inversion
$i \colon \C_1 \to \C_1$ (a strict bimorphism in $\Aa$) such that \emph{(v)} $r = s \circ i$, $s = r \circ i$ and \emph{vi)} for each $a \in \C_1$ we have
\[
m(i(a), a) = \id_{s(a)}, \ m(a, i(a)) = \id_{r(a)}. 
\]

For an internal groupoid we can summarize the structural maps in terms of a sequence:
\[
\xymatrix{
\C_2 \ar@{->>}[r]^{m} & \C_1 \ar@{>->>}[r]^{i} & \C_1 \ar@{->>}@< 2pt>[r]^-{r,s} \ar@{->>}@<-2pt>[r]^{} & \C_2 \ar@{>->}[r]^{u} & \C_1.
}
\]

\textbf{Notation:} Since we will be dealing with groupoids extensively we need to introduce a more convenient notation which will also correspond to the notation used in most papers on the subject. First we write $\G = (\Gop, \Gmor)$ for a
given (semi-)groupoid. Here $\Gop$ will denote the objects which are assumed to form a set by definition and by $\Gmor$ the arrows. Denote by $\Gpull$ the composable arrows, i.e. the pullback $\G \ast \G = \{(\gamma, \eta) \in \G \times \G : s(\gamma) = r(\eta)\}$. We keep the above notation for the 
structural maps, except that we make our lives easier by writing $\G$ for the arrows and denoting the inversion $i \colon \G \to \G$ by $\gamma \mapsto \gamma^{-1}$ as
well as the multiplication $m \colon \Gpull \to \G$ by $(\gamma, \eta) \mapsto \gamma \cdot \eta$ (and sometimes even leave out the $\cdot$ altogether). 

\textbf{The smooth category:} In the subsections below we introduce and study various specific categories which are internal to the category of smooth manifolds, e.g. the category of Lie groupoids.
The purpose of this is to establish a framework for the consideration of index theory problems on a class of manifolds much broader than the class of compact manifolds without boundary (i.e. closed manifolds).
We note however that the \emph{ambient category} of smooth manifolds we consider here consists of topological manifolds which are second countable, locally compact
topological spaces which need not be Hausdorff, may possess a boundary or corners and which are endowed with a smooth structure.
We will explicitly state whenever the manifold under consideration is Hausdorff or is without corners (i.e. is a closed smooth manifold). 
The \emph{smooth structure} on a given manifold with corners $M$ is always obtained by pulling back the class of smooth functions $C^{\infty}(U)$, defined on a fixed open neighborhood $U$ (without corners) of $M$, along the inclusion of $M$ into $U$.
We denote the ambient smooth category by $\Cinfty$ with the obvious choice of structure preserving maps.


\subsection{Lie groupoids}

We consider the category consisting of Lie groupoids as the objects and so-called \emph{correspondences} or \emph{generalized morphisms} as the arrows. 
These types of functors up to isomorphism between Lie groupoids were introduced in the pioneering work of Hilsum and Skandalis, \cite{HS}. 
We direct the reader to the main motivating applications of groupoids and algebroids in Section \ref{desing}.

We should emphasize at this point that the mere definition of a groupoid is exceedingly simple: They can be viewed as generalized
groups where the operation of composition is only partially defined.
More importantly groupoids generalize simultaneously such concepts as groups, singular spaces, sets, equivalence relations and foliations.
The driving force behind the concept is the correct notion of isomorphism inside a suitable category of groupoids. 
This notion of isomorphism is described by \emph{Morita equivalence} and we will see later how this name is motivated.

\begin{Def}
The tuple $(\Gop, \Gmor, r, s, m, u, i)$ is a \emph{Lie groupoid} iff $\Gop, \Gmor$ are smooth manifolds, all maps
are smooth and $s$ is a submersion. 
\label{Def:CLie}
\end{Def} 

Note that $r$ must be a submersion by the rule $r = s \circ i$. We use the shorthand $\G \rightrightarrows \Gop$ from now on to denote a (Lie) groupoid.
Since the ambient category consists of smooth manifolds possibly with corners the class of \emph{strict epimorphisms} in our case consists of submersions of manifolds with corners.

\begin{Exa}

The following are examples of Lie groupoids. 

\begin{itemize}
\item A smooth manifold $X \rightrightarrows X$ as the trivial set groupoid where the space of units is the same as the space of objects and the arrows are trivial. 

\item A smooth bundle $\pi \colon E \to X$ over a smooth manifold with $s = \pi = r$ and the fibers have a group structure with fiberwise defined composition.
As a particular case consider the tangent bundle $TX$ with addition in the fibers.

\item Lie groups are in particular Lie groupoids where the space of units consists only of the identity. 

\item Any equivalence relation on a set is a groupoid. In particular cases this can have more structure and be considered as Lie groupoid.
We will see many examples of this sort in the present work. If a group $G$ acts on a manifold $M$ the quotient space $M / G$ often turns out to be
a badly behaved object. In such a case it is preferable to consider the equivalence relation itself which in specific cases turns out to be a well-behaved Lie groupoid. 
\end{itemize}

\label{Exa:grpds}
\end{Exa}

\begin{Def}
Let $\G \rightrightarrows \Gop$ be a Lie groupoid. A $C^{\infty}$-manifold $Z$ is called a \emph{right $\G$-space} if there is a smooth map
$q \colon Z \to \Gop$ which is called \emph{charge map} and a smooth map $\alpha \colon Z \ast_r \G = \{(z, \gamma) : q(z) = r(\gamma)\} \to Z, (z, \gamma) \mapsto \alpha(z, \gamma) = z \cdot \gamma \in Z$
such that the following conditions hold

\emph{i)} $q(z \cdot \gamma) = s(\gamma), \ (z, \gamma) \in Z \ast_r \G$. 

\emph{ii)} $z \cdot (\gamma \cdot \eta) = (z \cdot \gamma) \cdot \eta$ for $(z, \gamma) \in Z \ast_r \G, \ (\gamma, \eta) \in \Gpull$, 

\emph{iii)} $z \cdot \id_{q(z)} = z$,

A right $\G$-space $(Z, \alpha, q)$ is called \emph{$\G$-fibered} if $q$ is a surjective submersion. 

The right action is \emph{free} if $\forall_{z \in Z} \ z \cdot \gamma = z \Rightarrow \gamma = \id_{q(z)}$
and \emph{proper} if the map $Z \ast_r \G \to Z \times Z, \ (z, \gamma) \mapsto (z, z \cdot \gamma)$ is proper. 

A free and proper action is called \emph{principal}. 

A \emph{left $\G$-space} is a right $\G^{op}$-space where $\G^{op}$ denotes the opposite category of $\G$. 

\label{Def:act}
\end{Def}

For later reference we give more details on the notion of a \emph{proper action}. 
Given a Lie groupoid $\G \rightrightarrows \Gop$ let $(Z, \alpha, q)$ be a $\G$-space. 
Then the \emph{semi-direct product groupoid} $Z \rtimes \G \rightrightarrows Z$ is the groupoid with the set of object $(Z \rtimes \G)^{(0)} = Z \ast_{\alpha} \G$, the structural maps
\[
r(z, \gamma) = z, \ s(z, \gamma) = z \cdot \gamma
\]

and composition 
\[
(z, \gamma) \cdot (w, \eta) = (z, \gamma \cdot \eta), \ w = z \cdot \gamma. 
\]

\begin{Def}
A Lie groupoid $\G \rightrightarrows \Gop$ is \emph{proper} if the map $(r, s) \colon \G \to \Gop \times \Gop$ is proper (i.e. the inverse image of every compact subset of $\Gop \times \Gop$ is compact in $\G$). 
\label{Def:proper}
\end{Def}

\begin{Thm}
Let $G \rightrightarrows \Gop$ be a Lie groupoid. Then $\G$ is proper if and only if $(r, s)$ is closed and for each $x \in \Gop$ the isotropy is quasi-compact. 
\label{Thm:proper}
\end{Thm}

\begin{proof}
The proof is from Tu, \cite{T3}, Proposition 2.10.
From the diffeomorphism $\G_x^x \cong \G_x^y$ if $\G_x^y \not= \emptyset$ we obtain the assertion by the following topological result: If $X, Y$ are topological spaces and $f \colon X \to Y$ is continuous, then $f$ is proper if and only if $f$ is closed and $f^{-1}(y)$ is quasi-compact for each $y \in Y$. 
\end{proof}

The previous result helps us also to further characterize the properness of groupoid actions, cf. \cite{T3}, Proposition 2.14.
\begin{Prop}
Let $Z$ be a right (resp. left) $\G$-space, then the following conditions are equivalent:

\emph{i)} The space $Z$ is a proper $\G$-space.

\emph{ii)} The semi-direct product groupoid $Z \rtimes \G$ (resp. $\G \ltimes Z$) is a proper Lie groupoid. 

\emph{iii)} The map $Z \ast \G \to Z \times Z, \ (z, \gamma) \mapsto (z, z \cdot \gamma)$ is closed and the stabilizers $(Z \rtimes \G)_z^z$ are quasi-compact for each $z \in Z$. 

\label{Prop:proper}
\end{Prop} 

\begin{proof}
We obtain \emph{ii)} $\Leftrightarrow$ \emph{iii)} immediately from Theorem \ref{Thm:proper}. Also \emph{i)} $\Leftrightarrow$ \emph{ii)} is immediate by the definition of source / range map in the groupoid
$Z \rtimes \G$. 
\end{proof}

\textbf{Strict category:}

We denote by $\LG$ the category with objects the Lie groupoids and arrows between objects given by \emph{strict morphisms}.

\begin{Def}
A \emph{strict morphism} of two Lie groupoids $F \colon \G \to \H$ is a tuple $F = (f, \fop)$ such that the diagram
\[
\xymatrix{
\G \ar@<-.5ex>[d] \ar@<.5ex>[d] \ar[r]^{f} & \H \ar@<-.5ex>[d] \ar@<.5ex>[d] \\
\Gop \ar[r]^{\fop} & \Hop 
}
\]

commutes and 
\begin{align}
& f(\gamma \cdot \eta) = f(\gamma) \cdot f(\eta), \ \forall \ (\gamma, \eta) \in \Gpull, \label{LG1} \\
& f(\gamma^{-1}) = f(\gamma)^{-1}, \ \gamma \in \G. \label{LG2}
\end{align}
\label{Def:strict}
\end{Def}

\begin{Rem}
\emph{(1)} Note that via $s = r \circ i$ and $r = s \circ i$ it is enough to require 
\[
\xymatrix{
\G \ar[d]_{\cdot} \ar[r]^{f} & \H \ar[d]_{\cdot} \\
\Gop \ar[r]^{\fop} & \Hop
}
\]

to commute, where $\cdot = s$ or $r$. 

\emph{(2)} If $(\gamma, \eta) \in \Gpull$, then \eqref{LG1}, \eqref{LG2} and $s = r \circ i$ imply that $(f(\gamma), f(\eta)) \in \Hpull$. 

\label{Rem:strict}
\end{Rem}

\textbf{Bibundle category:}

We will introduce the category $\LG_b$ consisting of objects the Lie groupoids and arrows between objects the generalized morphisms which are isomorphism classes of \emph{bibundle correspondences}. 

\begin{Def}
Let $\G \rightrightarrows \Gop$ and $\H \rightrightarrows \Hop$ be Lie groupoids. A \emph{bibundle correspondence} from $\H$ to $\G$ 
is a triple $(Z, p, q)$ where $(Z, \alpha, p)$ is a left $\H$-space and $(Z, \beta, q)$ is a right $\G$-space such that 

\emph{i)} the right action is principal and $Z$ is $\G$-fibered, 

\emph{ii)} the map $p$ induces a diffeomorphism $Z / \G \iso \Hop$, 

\emph{iii)} the actions of $\H  \ \rotatebox[origin=c]{-90}{$\circlearrowleft$}\ Z$ and $Z \ \rotatebox[origin=c]{90}{$\circlearrowright$}\ \G$ commute. 

An \emph{equivalence bibundle correspondence} of $\H$ and $\G$ (also called \emph{Morita equivalence} and denoted by $\H \simM \G$) is a triple $(Z, p, q)$ 
which is a bibundle correspondence from $\H$ to $\G$ such that

\emph{iv)} the left action is principal and $Z$ is $\H$-fibered,

\emph{v)} $q$ induces a diffeomorphism $\H / Z \iso \Gop$. 

\label{Def:gen}
\end{Def}

Additionally, we need to know how to compose morphisms in our category, hence we recall the definition of the \emph{generalized tensor product}. 
Given two generalized morphisms 
\[
\begin{tikzcd}[every label/.append style={swap}]
\H \ar[d, shift left] \ar[d] \ar[symbol=\circlearrowleft]{r} & \ar{dl}{p_1} Z_1 \ar{dr}{q_1} & \ar[symbol=\circlearrowright]{l} \G \ar[d, shift left] \ar[d] \\
\Hop & & \Gop
\end{tikzcd}
\]

and
\[
\begin{tikzcd}[every label/.append style={swap}]
\G \ar[d, shift left] \ar[d] \ar[symbol=\circlearrowleft]{r} & \ar{dl}{p_2} Z_2 \ar{dr}{q_2} & \ar[symbol=\circlearrowright]{l} \K \ar[d, shift left] \ar[d] \\
\Gop & & \Kop
\end{tikzcd}
\]

Endow the fiber-product $Z_1 \times_{\Gop} Z_2$ with the right $\H$-action $\gamma \colon (z_1, z_2) \mapsto (z_1 \gamma, \gamma^{-1} z_2)$ and set 
$Z_1 \circledast Z_2 = (Z_1 \times_{\G} Z_2) / \G$, the orbit space of the action. 
We obtain a generalized morphism
\[
\begin{tikzcd}[every label/.append style={swap}]
\H \ar[d, shift left] \ar[d] \ar[symbol=\circlearrowleft]{r} & \ar{dl}{\tilde{p}} Z_1 \circledast Z_2 \ar{dr}{\tilde{q}} & \ar[symbol=\circlearrowright]{l} \K \ar[d, shift left] \ar[d] \\
\Hop & & \Kop
\end{tikzcd}
\]

with left-action $\eta [z_1, z_2]_{\G} = [\eta z_1, z_2]_{\G}$ and right-action $[z_1, z_2]_{\G} \kappa = [z_1, z_2 \kappa]_{\G}$. 
The charge maps are given by $\tilde{p}([z_1, z_2]) = p_1(z_1), \ \tilde{q}([z_1, z_2]_{\G}) = q_2(z_2)$. 

This composition of bibundle correspondences has a left and right unit, namely the trivial generalized morphism $\G  \ \rotatebox[origin=c]{-90}{$\circlearrowleft$}\ \G \ \rotatebox[origin=c]{90}{$\circlearrowright$}\ \G$. 

\begin{Def}
The category $\LG_b$ consists of Lie groupoids as objects with isomorphism classes of bibundle correspondences as arrows between objects and the generalized tensor product $\circledast$ as composition of arrows. 
We refer to isomorphism class of bibundle correspondences as \emph{generalized morphisms} of Lie groupoids. 
\label{Def:LGb}
\end{Def}

The isomorphisms in this category are precisely given by Morita equivalences of Lie groupoids. 

\begin{Prop}
A bibundle correspondence $\G  \ \rotatebox[origin=c]{-90}{$\circlearrowleft$}\ Z \ \rotatebox[origin=c]{90}{$\circlearrowright$}\ \H$ is a Morita equivalence if and only if the class $[Z]$ 
is invertible as an arrow in $\LG_b$. 
\label{Prop:Morita}
\end{Prop}

\begin{proof}
We refer to \cite{L} and \cite{MRW}. 
\end{proof}

\begin{Thm}
The category of Lie groupoids with strict morphisms is included in the category of Lie groupoids with generalized morphisms via a functor $j_b \colon \LG \to \LG_b$. 
\label{Thm:incl}
\end{Thm}

\begin{proof}
We need to study the map on morphisms.

Given a strict morphism
\[
\xymatrix{
\G \ar@<-.5ex>[d] \ar@<.5ex>[d] \ar[r]^{f} & \H \ar@<-.5ex>[d] \ar@<.5ex>[d] \\
\Gop \ar[r]^{\fop} & \Hop 
}
\]

we obtain a pullback diagram
\[
\xymatrix{
\ar[d]_{\pi_1} \Gop _{\fop}\times_{r} \H \ar[r]^-{\pi_2} & \H \ar[d]_{r} \\
\Gop \ar[r]^{\fop} & \Hop 
}
\]

This yields the generalized morphism
\[
\begin{tikzcd}[every label/.append style={swap}]
\G \ar[d, shift left] \ar[d] \ar[symbol=\circlearrowleft]{r} & \ar{dl}{\pi_1} \Gop _{\fop} \times_{r} \H \ar{dr}{r \circ \pi_2} & \ar[symbol=\circlearrowright]{l} \H \ar[d, shift left] \ar[d] \\
\Gop & & \Hop 
\end{tikzcd}
\]

The left action $\G \ \rotatebox[origin=c]{-90}{$\circlearrowleft$}\ \Gop _{\fop} \times_{r} \H$ is given by
\[
\gamma \cdot (s(\gamma), \eta) = (r(\gamma), f(\gamma) \cdot \eta), \ \gamma \in \G, \ \eta \in \H^{\fop(s(\gamma))}. 
\]

Finally, the right action $\Gop _{\fop} \times_{r} \H \ \rotatebox[origin=c]{90}{$\circlearrowright$} \ \H$ is given by
\[
(x, \eta) \cdot \tilde{\eta} = (x, \eta \cdot \tilde{\eta})
\]

and this action is principal. 
\end{proof}


\subsection{Lie algebroids}

In this section we recall the definition of the category of Lie algebroids. We refer to \cite{M} for a more exhaustive presentation of the theory.

\begin{Def}
Let $\pi \colon E \to M$ be a vector bundle over a smooth manifold $M$ together with a vector bundle map $\varrho \colon E \to \TM$ 
\[
\xymatrix{
E \ar[d]_{\pi} \ar[r]^{\varrho} & \TM \ar[dl] \\
M & 
}
\]
with Lie bracket on $\Gamma(E)$ which we denote henceworth by $[\cdot, \cdot]_E$. If we have
\begin{align*}
\varrho \circ [V, W]_E = [\varrho \circ V, \varrho \circ W]_{\Gamma(\TM)} 
\end{align*}

and
\begin{align*}
[V, f \cdot W]_E = f [V, W]_E + ((\varrho \circ V) f) \cdot W, \ \forall \ V, W \in \Gamma(E), f \in C^{\infty}(M).
\end{align*}

Then $E$ is called a \emph{Lie algebroid} with \emph{anchor} $\varrho$. 
\label{Def:LieAlg}
\end{Def}

\begin{Exa}

\emph{i)} A Lie algebra is a particular example of a Lie algebroid where $M = \{e\}$ is a single point. 

Given a tangent bundle $\pi \colon TM \to M$ over a $C^{\infty}$ manifold $M$, then $TM$ is a Lie algebroid where
the anchor is the identity. 

\emph{ii)} A more non-trivial example is provided by a smooth and projective foliation $\F \to M$ which is a Lie algebroid with injective anchor map.
More examples of Lie algebroid will occur throughout.

\label{Exa:algbroids}
\end{Exa}

\textbf{Construction of Lie algebroids:} Given a Lie groupoid $\G \rightrightarrows \Gop$ we can view a Lie algebroid $\A(\G) \to M$ as the infinitesimal object associated to the groupoid. 

Let $\G \rightrightarrows \Gop = M$ be a Lie groupoid.
For $\gamma \in \G$ we fix the right multiplication map
\[
R_{\gamma} \colon \G_{r(\gamma)} \to \G_{s(\gamma)}, \ \eta \mapsto \eta \cdot \gamma.
\] 


\begin{Def}
Define $T^s \G := \ker(ds)$ the \emph{$s$-vertical tangent bundle} as a sub-bundle of $T\G$. 
Denote by $\Gamma(T^s \G)$ the smooth sections and define $\Gamma_R(T^s \G)$ as the sections $V$ such that
\[
V(\eta \gamma) = (R_{\gamma})_{\ast} V_{\eta} \ \text{for} \ (\eta, \gamma) \in \Gpull.
\]
\label{Def:invVF}
\end{Def}

\begin{Rem}
If $V \in \Gamma_R(T^s \G)$, then $V$ is determined by the values on the base $M$, since
\[
V(\gamma) = V(u(r(\gamma)) \cdot \gamma) = (R_{\gamma})_{\ast} V(u(r(\gamma))). 
\]

On the other hand if $V$ is $s$-vertical such that $V(\gamma) = (R_{\gamma})_{\ast} V(u(r(\gamma)))$ then for $(\gamma, \eta) \in \Gpull$ we have
\begin{align*}
V(\gamma \eta) &= V(u(r(\gamma \eta)) \gamma \eta) = V(u(r(\gamma)) \gamma \eta) \\
&= (R_{\gamma \eta})_{\ast} V(u(r(\gamma))) = (R_{\eta} \circ R_{\gamma})_{\ast} V(u(r(\gamma))) \\
&= (R_{\eta})_{\ast} V(\gamma).
\end{align*}

\label{Rem:invVF}
\end{Rem}

The Lie-algebroid associated to $\G$ is the pullback
\[
\xymatrix{
\A(\G) \ar[d] \ar[r]^{u^{\ast}} & T^s \G \ar[d]_{\pi_{|s}} \\
M \ar[r]^{u} & \G.
}
\]

Therefore $\A(\G) = \{(V, x) | ds(V) = 0, u(x) = 1_x = \pi(V)\}$. 

\begin{Prop}
There is a canonical isomorphism of Lie-algebras $\Gamma_R(T^s \G) \cong \Gamma(\A(\G))$. 
\label{Prop:invVF}
\end{Prop}

\begin{proof}
The map $\Gamma_R(T^s \G) \to \Gamma(\A(\G))$ is defined $V \mapsto V \circ u$.
Then $\Gamma(\A(\G)) \to \Gamma_R(T^s \G)$ is given by $V \mapsto \tilde{V}$ with
$\tilde{V}(\gamma) := (R_{\gamma})_{\ast} V(r(\gamma))$. 
Then we see that $\tilde{V}(\eta \gamma) = (R_{\eta})_{\ast} V(r(\eta))$ and hence $\tilde{V} \in \Gamma_R(T^s \G)$.
And this gives a linear isomorphism. 

Now $\Gamma(T^s \G)$ is closed with regard to the Lie-bracket $\Gamma(T\G)$, since it contains vector fields which 
are tangent to the $s$-fibers. Since local diffeomorphisms preserve Lie-bracket we also see that $\Gamma_R(T^s \G)$ 
is closed under Lie bracket. 
\end{proof}

The set of smooth sections $\Gamma(\A(\G))$ is a $C^{\infty}(M)$-module with the module operation $f \cdot V = (f \circ r) \cdot V$
with $f \in C^{\infty}(M)$.

\begin{Prop}
Let $\A(\G)$ be given as above and define $\varrho \colon \A(\G) \to \TM$ by $\varrho := dr \circ u^{\ast}$.
Then $(\A(\G), \varrho)$ furnishes a Lie algebroid.
\label{LieAlg}
\end{Prop}

\begin{proof}
We have
\begin{align*}
\widetilde{[V, fW]} &= [\tilde{V}, (f \circ r) \tilde{W}] \\
&= (f \circ r) [\tilde{V}, \tilde{W}] + \tilde{V} (f \circ r) \tilde{W} \\
&= \widetilde{f[V, W]} + \tilde{V} (f \circ r) \tilde{W}.
\end{align*}

Since $r \circ R_{\gamma} = r$ for each $\gamma \in \G$ and $r$ is a submersion we find $V' \in \Gamma(\TM)$
such that
\[
V'(f) \circ r = \tilde{V}(f \circ r).
\]

From $r_{\ast} \tilde{V}(\gamma) = r_{\ast} (R_{\gamma})_{\ast} V(r(\gamma)) = r_{\ast} V(r(\gamma))$ we obtain
\[
\varrho_{\ast} = r_{\ast|u(x)} = T_{u(x)} \G_x \to T_x M.
\]

Therefore $V' = \varrho(V)$ and this yields
\[
[V, f W] = f [V, W] + \varrho(V)(f) \cdot W
\]

as well as
\[
\varrho([V, W]) = \widetilde{[V, W]} = [\varrho(V), \varrho(W)].
\]

This ends the proof.
\end{proof}

\begin{Rem}
A Lie algebroid is said to be \emph{integrable} if we can find an associated ($s$-connected) Lie groupoid.
Not every Lie algebroid is integrable thus Lie's third theorem (for finite Lie groups) fails to carry over to algebroids.
\label{Rem:int}
\end{Rem}

\textbf{Strict category:} We denote by $\LA$ the category of Lie algebroids with arrows between objects defined as follows. 

\begin{Def}
A \emph{strict morphism} between two Lie algebroids $\Phi \colon (\pi_1 \colon \A_1 \to M_1, \varrho_1) \to (\pi_2 \colon \A_2 \to M_2, \varrho_2)$ is given by a tuple $\Phi = (\varphi, \varphiop)$ such that the following diagram commutes
\[
\xymatrix{
M_1 \ar[r]^{\varphiop} & M_2 \\
\A_1 \ar[u]^{\pi_1} \ar[d]_{\varrho_1} \ar[r]^{\varphi} & \A_2 \ar[d]_{\varrho_2} \ar[u]^{\pi_2} \\
TM_1 \ar[r]^{d \varphiop} & TM_2 
}
\]

meaning 
\begin{align}
& \varphiop \circ \pi_1 = \pi_2 \circ \varphi, \label{LA1} \\
& d \varphi \circ \varphi_1 = \varrho_2 \circ \varphi. \label{LA2} 
\end{align}

Additionally, $\Phi$ preserves the anchors and induces a Lie algebra homomorphism on the sections $\Gamma(\A_1) \to \Gamma(A_2)$. 
\label{Def:strict2}
\end{Def}

For later use we recall here the definition of a Lie subalgebroid, cf. \cite{M}, p. 164. 
\begin{Def}
Given a manifold $M$ and a submanifold $N \subset M$ with algebroid $(\A, \varrho)$ defined over $M$.
Then a Lie algebroid $(\tilde{\A}, \tilde{\varrho})$ over $N$ is a \emph{subalgebroid} of $\A$ iff
$\tilde{\A} \subset \A_{|N}$ is a subbundle equipped with a Lie algebroid structure s.t. the inclusion
$\tilde{\A} \hookrightarrow \A_{|N}$ is a Lie algebroid morphism. 
\label{Def:subbroid}
\end{Def}

\begin{Rem}
We recall below the category equivalence of Lie algebroids and linear Poisson structures on manifolds. The notion of a structure preserving 
map between Lie algebroids is therefore the same as the structure preserving maps between linear Poisson structures. The latter are considerably easier
to define.
\label{Rem:strict}
\end{Rem}

\textbf{Poisson category:} We define the category of \emph{integrable Poisson manifolds} which is the classical analogue of
the category of $C^{\ast}$-algebras with isomorphism classes of correspondence bimodules that we will introduce below.
When we consider a bibundle correspondence between Lie groupoids, we may ask what a suitable notion
of correspondence of the associated Lie algebroids should be.
Specifically, we would like to define a notion of correspondence for Lie algebroids which makes the association
$\G \mapsto \A(\G)$ of a Lie groupoid to its Lie algebroid functorial.
We rely on the work of Xu who defined the notion of Morita equivalence of Poisson manifolds \cite{Xu2} as well as on \cite{L0}, \cite{L} for the definition of the category of integrable Poisson manifolds which
are duals of Lie algebroids. First we recall the category equivalence between linear Poisson structures on manifolds and Lie algebroids.
By a linear Poisson structure we mean that the Poisson bracket of two fiberwise linear functions is again linear. 
\begin{Thm}
Given a smooth vector bundle $E \to M$ which is endowed with a linear Poisson structure, then $E^{\ast} \cong \A$ for 
a Lie algebroid $\A \to M$.
Conversely, the dual of a Lie algebroid has a canonical linear Poisson structure. In other words we have a category equivalence:
\[
\{\text{cat. of linear Poisson structures on vector bundles}\} \cong \{\text{cat. of Lie algebroids}\}. 
\]
\label{Thm:catequ}
\end{Thm}

Let $P, Q$ be Poisson manifolds. A \emph{Weinstein dual pair} $Q \leftarrow S \rightarrow P$ consists of 
a symplectic manifold $S$ and Poisson maps $q \colon S \to Q, \ p \colon S \to P^{-}$ such that 
$\{q^{\ast} f, p^{\ast} g\} = 0, \ f \in C^{\infty}(Q), \ g \in C^{\infty}(P)$. If $Q \leftarrow S_i \rightarrow P, \ i = 1,2$ are two Weinstein dual pairs, then they are defined to be \emph{isomorphic}
if there is a symplectomorphism $\varphi \colon S_1 \to S_2$ such that $q_2 \varphi = q_1, \ p_2 \varphi = p_1$. 
A \emph{regular} dual pair is a dual pair as a above for which $q$ is a surjective submersion and $p,q$ are both complete Poisson maps.
The category $\LA_b$ consists of objects given by dual Lie algebroids $\A^{\ast}(\G)$ associated to arbitrary Lie groupoids
$\G \rightrightarrows \Gop$. The arrows are isomorphism classes of Weinstein dual pairs of the type $\A^{\ast}(\H) \leftarrow T^{\ast} Z \rightarrow \A^{\ast}(\G)$
induced by a correspondence bibundle $\H  \ \rotatebox[origin=c]{-90}{$\circlearrowleft$}\ Z \ \rotatebox[origin=c]{90}{$\circlearrowright$}\ \G$ of Lie groupoids.

\subsection{Symplectic groupoids}

In the study of quantization theory of singular manifolds we need to consider Lie groupoids with further symplectic structure. We introduce the category of \emph{symplectic groupoids} $\SG_b$ as a subcategory of $\LG_b$ which consists of symplectic Lie groupoids together with isomorphism classes of symplectic bibundle correspondences. 

\begin{Def}
\emph{i)} A \emph{symplectic groupoid} $(\Gamma, \omega)$ is a Lie groupoid $\Gamma$ such that the space of morphisms $\Gamma = \Gamma^{(1)}$ is a symplectic manifold with symplectic $2$-form 
$\omega$ such that the graph of $\Gamma^{(2)}$ is a Lagrangian submanifold of $\Gamma \times \Gamma \times \Gamma^{-}$ with respect to $\omega \oplus \omega^{-}$ where $(\Gamma^{-}, \omega^{-}) = (\Gamma, -\omega)$.

\emph{ii)} An action $\alpha$ of a symplectic groupoid $(\Gamma, \omega)$ on a symplectic manifold $(S, \omega_S)$ is \emph{symplectic} if the graph of the action $\Gr(\alpha) \subset \Gamma \times S \times S^{-}$ is a Lagrangian submanifold with regard to $\omega_S \oplus \omega_{S^{-}}$ on $S \times S^{-}$ where $\omega_{S^{-}} = -\omega_S$. 
\label{Def:SGb}
\end{Def}

The above definition entails that the groupoid multiplication in a symplectic groupoid corresponds to a canonical relation. 
In the same way the symplectic action corresponds to a canonical relation as well. 

\emph{Morphisms in $\SG_b$:} Let $\Gamma$ and $\Sigma$ be symplectic groupoids. Then a \emph{symplectic bibundle} $S \in (\Gamma, \Sigma)$ in $\SG_b$ consists of two symplectic actions $\Gamma  \ \rotatebox[origin=c]{-90}{$\circlearrowleft$}\ S \ \rotatebox[origin=c]{90}{$\circlearrowright$}\ \Sigma$ on a given symplectic space $S$ where the right action is principal.

\emph{Composition:} Given two symplectic bibundles $\Gamma_1  \ \rotatebox[origin=c]{-90}{$\circlearrowleft$}\ S_1 \ \rotatebox[origin=c]{90}{$\circlearrowright$}\ \Sigma$ and $\Sigma  \ \rotatebox[origin=c]{-90}{$\circlearrowleft$}\ S_2 \ \rotatebox[origin=c]{90}{$\circlearrowright$}\ \Gamma_2$. Then there is a generalized tensor product $\Gamma_1  \ \rotatebox[origin=c]{-90}{$\circlearrowleft$}\ S_1 \circledast_{\Sigma} S_2 \ \rotatebox[origin=c]{90}{$\circlearrowright$}\ \Gamma_2$ which is a morphism in $\SG_b$ as can be checked, cf. \cite{L}. 

\emph{Morita equivalence ($\simM$):} A \emph{Morita equivalence} between two symplectic groupoids $\Sigma$ and $\Gamma$ is implemented by a symplectic bibundle $S \in (\Sigma, \Gamma)$ which is biprincipal. 





It is straightforward to define a suitable notion of isomorphism between two symplectic bibundles, i.e. a diffeomorphism
which is at the same time a symplectomorphism which is compatible with the actions. 
As stated previously the category $\SG_b$ therefore is defined to consist of symplectic groupoids as the objects and isomorphism classes of symplectic bibundles
as the arrows. The composition of the arrows is fascilitated by the generalized tensor product and the units are induced by the canonical symplectic bibundles $\Sigma  \ \rotatebox[origin=c]{-90}{$\circlearrowleft$}\ \Sigma \ \rotatebox[origin=c]{90}{$\circlearrowright$}\ \Sigma$,
where $\Sigma$ is a symplectic groupoid with the obvious left and right actions.
Note that this makes $\SG_b$ into a subcategory of $\LG_b$ consisting of Lie groupoids as the objects together with isomorphism classes of bibundles as
the arrows between objects. It also holds that two symplectic groupoids are isomorphic objects in the category $\SG_b$ if and only if they are Morita equivalent.
\begin{Prop}
A symplectic bibundle $S \in (\Gamma, \Sigma)$ is a Morita equivalence if and only if its isomorphism class $[S]$ is invertible
as an arrow in $\SG_b$. 
\label{Prop:Morita2}
\end{Prop}

\begin{proof}
We refer to \cite{L0} and \cite{L} for a proof. 
\end{proof}

\begin{Rem}
The category $(\LA_b, \circledcirc)$ is equivalent to the full subcategory of $\SG_b$ with objects the $s$-connected and $s$-simply connected symplectic groupoids. 

Where we denote by $\circledcirc$ and $\circledast$ the corresponding generalized tensor products, based on \cite{L}. 
In general note the functorial relationship between the category of Lie groupoids and the category of symplectic groupoids 
\[
\xymatrix{
(\LG_b, \circledast) \ar[d]_{\A} & \\
(\P_b, \circledcirc) \ar[r]^{\Gamma} & (\SG_b, \circledast) \ar@{^{(}->}[ul]^{incl.} 
}
\]

\label{Rem:Pcategory}
\end{Rem}

\subsection{$C^{\ast}$-algebras}

\label{Cstarb}

We consider $C^{\ast}$-algebras over Lie groupoids which are defined in the following way.

\begin{Def}
The involutive algebra $(C_c^{\infty}(\G), \ast)$ has the convolution product
\[
(f \ast g)(\gamma) = \int_{\G_x} f(\gamma \eta^{-1}) g(\eta) \,d\mu_x(\eta) 
\]

and involution $f^{\ast}(\gamma) = \overline{f(\gamma^{-1})}$. 

Then the $\ast$-representation $\pi_x \colon C_c^{\infty}(\G) \to L^2(\G_x)$ is defined by
\[
(\pi_x(f) \xi)(\gamma) = \int f(\eta) \xi(\eta^{-1} \gamma) \,d\mu_x(\eta), \gamma \in \G_x, \xi \in L^2(\G_x).
\]

Define the $C^{\ast}$-algebra of $\G$ by $C_r^{\ast}(\G) := \overline{C_c^{\infty}(\G)}^{\|\cdot\|_r}$. The norm is
\[
\|f\|_r := \sup_{x \in \Gop} \|\pi_x(f)\|_2, f \in C_c^{\infty}(\G).
\]

\label{Def:repr}
\end{Def}

We are also interested in the functorial relationship between a suitable category of $C^{\ast}$-algebras and the category of Lie groupoids. 
For a more detailed presentation of the necessary background to this theory we first refer to the famous work of Muhly-Renault-Williams, \cite{MRW} and the general functoriality as can be found in Landsman \cite{L}. 
We define by $C_b^{\ast}$ the category with objects the (separable) $C^{\ast}$-algebras and arrows between objects given by isomorphism classes of bimodule correspondences.

\begin{Def}
Let $\E$ be a Banach space such that $\E$ is endowed with a right Hilbert $B$-module structure and a non-degenerate $\ast$-homomorphism $\pi \colon A \to \L_B(\E)$.
Then $\E$ is called a \emph{bimodule correspondence} and we write $\E \in (A, B)$. 
\label{Def:bimodule}
\end{Def}


Let $\E_1$ be an $(A, B)$-bimodule correspondence and let $\E_2$ be a $(B, C)$-bimodule correspondence. 
Then $\E_1 \otimes \E_2$ has a canonical $(A, C)$-bimodule structure, with the inner product $\scal{\cdot}{\cdot} \colon \E_1 \otimes \E_2 \to A$ given by
\[
\scal{\xi_1 \otimes \xi_2}{\eta_1 \otimes \eta_2}_A := \scal{\xi_1}{ \eta \scal{\eta_1}{\xi_2}_B}_A.
\]

Define the equivalence relation $\sim$ on $\E_1 \otimes \E_2$ via 
\[
\xi b \otimes \eta \sim \xi \otimes b \eta, \ \xi \in \E_1, \ \eta \in \E_2, \ b \in B.
\]

We complete the quotient by this equivalence relation with regard to the induced $A$-valued norm 
\begin{align}
\E_1 \hat{\otimes}_B \E_2 &= \overline{\E_1 \otimes \E_2 / \sim}^{\| \cdot\|}. \label{Rieffel}
\end{align}

The Rieffel tensor product yields a Hilbert $(A, C)$-bimodule. 

 
Let $A$ be a $C^{\ast}$-algebra and let $\E$ be a left $A$ module, $\F$ be a right Hilbert $A$ module.
Define the maps $\theta_{x,y} \colon \E \to \F$ for given $x \in \F, \ y \in \E$ by 
$\theta_{x,y}(z) = x \scal{y}{z}_A$. 
We define the class of \emph{generalized compact operators} $\K(\E, \F)$ to be the closure of the span over the $\theta_{x,y}$, i.e. $\K(\E, \F) := \overline{\mathrm{span}\{\theta_{x,y} : x \in \F, y \in \E\}}$.
Note that $\K(\E, \F)$ is contained in the space of linear adjointable maps $\E \to \F$, i.e. $\K(\E, \F) \subset \L(\E, \F)$.
If $\F = A$ for $A$ a $C^{\ast}$-algebra we write $\K_{A}(\E) = \K(\E, A)$. 

\begin{Rem}
If $A = \Cc$ is the complex numbers and $\E = \H$ a complex Hilbert space we obtain that $\K_{\Cc}(\H) = \K(\H)$ are the compact operators on the Hilbert space $\H$.
In general the elements $\K(\E, \F)$ are not compact operators, hence the name generalized compact operators. 
\label{Rem:gencpt}
\end{Rem}

\begin{Def}
Let $A$ and $B$ be $C^{\ast}$-algebras. Then $A$ is \emph{Morita equivalent} to $B$ (written $A \simM B$) if there is 
a bimodule correspondence $\E \in (A, B)$ with the following properties:

\emph{(i)} the linear span of the range of $\scal{\cdot}{\cdot}_B$ is dense in $B$.

\emph{(ii)} The $\ast$-homomorphism $\pi \colon A \to \L_B(\E)$ is an isomorphism $A \cong \K_B(\E)$.  

\label{Def:MoritaCstar}
\end{Def}

The category $C_b^{\ast}$ consists of isomorphism classes of bimodule correspondences with composition given by the generalized tensor product. 
We refer to an isomorphism class $[\E]$ of bimodule correspondence $\E \in (A, B)$ as a \emph{generalized morphism} of $C^{\ast}$-algebras, written $A \dashrightarrow B$. 


The next proposition shows that Morita equivalence is the same as isomorphy for $C^{\ast}$-algebras in the category $C_b^{\ast}$.

\begin{Prop}
A bimodule correspondence of $C^{\ast}$-algebras $\E \in (A, B)$ is a Morita equivalence if and only if its isomorphism
class $[\E]$ is invertible as an arrow in $C^{\ast}$. 
\label{Prop:MoritaCstar}
\end{Prop}

\begin{proof}
See \cite{L}, Prop. 3.7. 
\end{proof}

\begin{Prop}
There is a canonical covariant functor of inclusion \ $\widehat{}_b \ \colon C^{\ast} \hookrightarrow C_b^{\ast}$. 
\label{Prop:inclCstar}
\end{Prop}

\begin{proof}
We only have to describe the inclusion on morphisms. Let $f \colon A \to B$ be a non-degenerate strict morphism. Then we have
the assignment the module structure $B \times B \to B, \ (b_1, b_2) \mapsto b_1 b_2$. The scalar product 
$\scal{}{}_B \colon B \times B \to B$ given by $(b_1, b_2) \mapsto b_1^{\ast} b_2$. 
Finally, the non-degenerate $\ast$-homomorphism $A \to \L_B(B)$ is defined by $a \mapsto (b \mapsto f(a) \cdot b)$.  
\end{proof}

\subsection{Functoriality}

The Muhly-Renault-Williams theorem states that Morita equivalent Lie groupoids yield (strongly) Morita equivalent corresponding $C^{\ast}$-algebras.
This result has been extended to relate the categories we have introduced by functorial correspondences. The statement can be generalized in the following Theorem,
based on the references \cite{L0}, \cite{L}, \cite{L2}, \cite{Xu}, \cite{Xu2} and \cite{MRW}. 

\begin{Thm}[Functoriality]
\emph{i)} There is a functorial correspondence $\LG_b \to \LA_b$ from the category of $s$-connected Lie groupoids to the category of integrable Lie algebroids given by 
$\G \mapsto \A^{\ast}(\G)$ on objects and $[\H  \ \rotatebox[origin=c]{-90}{$\circlearrowleft$}\ Z \ \rotatebox[origin=c]{90}{$\circlearrowright$}\ \G] \mapsto [\A^{\ast}(\H)  \leftarrow T^{\ast} Z \rightarrow \A^{\ast}(\G)]$. 
In particular if $\H \simM \G$ in $\LG_b$ then $\A^{\ast}(\H) \simM \A^{\ast}(\G)$ in $\LA_b$, i.e. the functor preserves Morita equivalence. 

\emph{ii)} There is a functorial correspondence $\LG_b \ni \G \mapsto C^{\ast}(\G) \in C_b^{\ast}$ and $[\H  \ \rotatebox[origin=c]{-90}{$\circlearrowleft$}\ Z \ \rotatebox[origin=c]{90}{$\circlearrowright$}\ \G] \mapsto [C^{\ast}(\H)  \ \rotatebox[origin=c]{-90}{$\circlearrowleft$}\ \E_Z \ \rotatebox[origin=c]{90}{$\circlearrowright$}\ C^{\ast}(\G)]$. 
In particular if $\H \simM \G$ then $C^{\ast}(\H) \simM C^{\ast}(\G)$, i.e. the functor preserves Morita equivalence. 
\label{Thm:functoriality}
\end{Thm}

\begin{proof}
\emph{i)} See e.g. \cite{LR}. 

\emph{ii)} For a proof of the Muhly-Renault-Williams theorem for locally compact groupoids (Morita equivalent groupoids induce Morita equivalent $C^{\ast}$-algebras) we refer to \cite{MRW} and for the functoriality assertion we refer to \cite{L}, \cite{L2} as well as \cite{MO}.  
\end{proof}

\begin{Rem}
Denote by $\CG$ the category of \emph{algebraic} groupoids with the arrows given by strict morphisms, i.e. the groupoids internal to the category $\Set$ with the (external)
axiom of choice. In particular all pullbacks in $\CG$ exist. Also denote by $\CG_b$ the category of groupoids with arrows given by 
bibundle correspondences. In this case we have not only the inclusion $\CG \hookrightarrow \CG_b$ but also an inclusion in the other direction.
Given a bibundle correspondence $\G  \ \rotatebox[origin=c]{-90}{$\circlearrowleft$}\ Z \ \rotatebox[origin=c]{90}{$\circlearrowright$}\ \H$, where $p \colon Z \to \Gop$ the charge map of the corresponding
left action of $\G$. Then by the external axiom of choice we can pick a section $\sigma \colon \Gop \to Z$, i.e. $p \circ \sigma = \id_{Z}$. 
We define the strict morphism $\Psi^{\sigma} \colon \G \to \H$ via $\Psi_0^{\sigma}(x) := (q \circ \sigma)(x)$ on the level of objects $x \in \Gop$ of the groupoid and
$\Psi^{\sigma}(\gamma) := \eta$, where $\eta$ is uniquely determined by the equality $\gamma \sigma(s(\gamma)) = \sigma(r(\gamma)) \Psi^{\sigma}(\gamma)$ for $\gamma \in \G$.
One can check that this yields in fact a strict groupoid morphism, cf. \cite{L}. Therefore the categories $\CG$ and $\CG_b$ are in fact equivalent.
Such an argument fails if the ambient category is $\Cinfty$, i.e. in the case of Lie groupoids, since no smooth external axiom of choice is available.
\label{Rem:EAC}
\end{Rem}


\section{Pseudodifferential operators}

\label{PsDos}

The goal of this section is two-fold: At first we need to introduce pseudodifferential operators on a smooth groupoid useful in index theory.
Secondly, pseudodifferential operators (and more generally convolution algebras of distributions) on groupoids provide a convenient setting for analysis on singular manifolds.

\subsection{Quantization}

Fix a Lie groupoid $\G \rightrightarrows M$. 

\begin{Def}
A \emph{right Haar system} $(\mu_x)_{x \in M}$ is a family of $C^{\infty}$-measures supported in $\G_x$ such that the following
conditions hold.

\emph{i)} For each $f \in C_c^{\infty}(\G)$ the assignment
\[
\Gop \ni x \mapsto \int_{\G_x} f \,d\mu_x 
\]

is smooth.

\emph{ii)} For each $f \in C_c^{\infty}(\G)$ and $\gamma \in \G$ the right transform $R_{\gamma}$ is measure-preserving, i.e.
\begin{align}
\int_{\G_x} f(\eta) \,d\mu_x(\eta) &= \int_{\G_y} f(\eta \cdot \gamma) \,d\mu_y(\eta), \ x = s(\gamma), y = r(\gamma) \label{I}.
\end{align}
\label{HaarS}
\end{Def}


Abusing notation we set $R_{\gamma} \colon C_c^{\infty}(\G_{s(\gamma)}) \to C_c^{\infty}(\G_{r(\gamma)})$ for the right transform on functions, i.e. $(R_{\gamma} f)(\eta) = f(\eta \gamma), \eta \in \G_{r(\gamma)}$.

\begin{Def}
A $\G$-operator is a family $P = (P_x)_{x \in \Gop}$ of continuous, linear operators $P_x \colon C_c^{\infty}(\G_x) \to C_c^{\infty}(\G_x)$ s.t. 
\[
P(f)(\gamma) := P_{s(\gamma)}(f_{s(\gamma)})(\gamma), f_{s(\gamma)} := f_{|\G_x}, s(\gamma) = x
\]

and the following invariance condition
\begin{align}
R_{\gamma} P_{s(\gamma)} &= P_{r(\gamma)} R_{\gamma}, \ \gamma \in \G. \label{II}
\end{align}
\label{Def:GOp}
\end{Def}

\begin{Rem}
Applying the Schwartz kernel theorem to the manifold $(\G_x, \mu_x)$ we obtain a family of integral kernels
$(k_x)_{x \in \Gop}$ with $k_x \in C^{\infty}(\G_x \times \G_x)'$ acting on smooth functions.
One can write
\[
(Pf)(\gamma) = \int_{\G_x} k_x(\gamma, \eta) f(\eta) \,d\mu_x(\eta), x = s(\gamma).
\]

In general the invariance condition yields
\begin{align*}
R_{\gamma}(Pf)(\eta)  = (Pf)(\eta \gamma) &= \int_{\G_x} k_x(\eta \gamma, \tilde{\eta}) f(\tilde{\eta}) \,d\mu_x(\tilde{\eta}), \ x = s(\eta \gamma) = s(\gamma) 
\end{align*}

and
\begin{align*}
P(R_{\gamma} f)(\eta) &= \int_{\G_y} k_y(\eta, \tilde{\eta}) f(\tilde{\eta} \gamma) \,d\mu_y(\tilde{\eta}) \\
&=_{(\tilde{\gamma} = \tilde{\eta} \gamma)} \int_{\G_x} k_y(\eta, \tilde{\gamma} \gamma^{-1})f(\tilde{\gamma})\,d\mu_x(\tilde{\gamma}).
\end{align*}

Hence by \eqref{II} 
\begin{align}
\forall_{\gamma \in \G} \ k_x(\eta \gamma, \tilde{\gamma}) = k_y(\eta, \tilde{\gamma} \gamma^{-1}) \ (x = s(\gamma), y = r(\gamma)). \tag{$*$} \label{*}
\end{align}

Setting $k(\beta) := k_{s(\gamma)}(\gamma, \eta)$ with $\beta = \eta \gamma^{-1}$ one can see that this is independent of the choice of $\gamma, \eta$.
To this end set $\tilde{\eta} \tilde{\gamma}^{-1} = \beta, \delta = \gamma^{-1} \tilde{\gamma}$ and check
\begin{align*}
k_{s(\tilde{\gamma})}(\tilde{\gamma}, \tilde{\eta}) &= k_{s(\delta)}(\tilde{\gamma}, \tilde{\eta}) =_{\eqref{*}} k_{r(\delta)}(\tilde{\gamma} \delta^{-1}, \tilde{\eta} \delta^{-1}) \\
&= k_{s(\gamma)}(\gamma, \eta).
\end{align*}

Hence $P$ is given by 
\[
P(f)(\gamma) = \int_{\G_x} k_P(\gamma \eta^{-1}) f(\eta) \,d\mu_x(\eta), \ x = s(\gamma).
\]
\label{Rem:GOp}
\end{Rem}

\textbf{Fiber preserving charts:} Let $\Omega \subset \G$ be an open subset, $W \subset \Rr^n$ be open, set $V_s := s(\Omega)$ and fix a diffeomorphism $\psi_s \colon W_s \times V_s \to \Omega$.
Then $\psi_s$ is called \emph{$s$-fiber preserving} if the diagram
\[
\xymatrix{
\Omega \ar[d]_{s} & \ar[l]_-{\psi_s} W \times s(\Omega) \ar[dl]_-{\pi_2} \\
s(\Omega)
}
\]

commutes. In other words for each $(x, w) \in V_s \times W_s$ we have $s(\psi_s(x, w)) = x$. 
Fix the notation $\Omega \sim W \times s(\Omega) \sim W \times V$ if there is such an $s$-fiber preserving diffeomorphism.
Denote by $\psi_s^{\ast} \colon C_c^{\infty}(W \times V) \to C_c^{\infty}(\Omega)$ the pullback over compactly supported
smooth functions. Then $P$ is called a \emph{smooth family} or \emph{$C^{\infty}$-family} if for each fiber preserving diffeomorphism $\psi_s$ as above
the operator $(\psi_s^{-1})^{\ast} \circ P \circ (\psi_s)^{\ast} \colon C_c^{\infty}(W \times V) \to C_c^{\infty}(W \times V)$
is a family of properly supported pseudodifferential operators, locally parametrized over $s(\Omega)$.

\textbf{Support condition:} Let $\mu \colon \G \ast_{s} \G \to \G, \ (\gamma, \eta) \mapsto \gamma \eta^{-1}$.
Where we set $\G \ast_s \G := \{(\gamma, \eta) \in \G \times \G : s(\gamma) = s(\eta)\}$.
Denote by $\supp(P) = \overline{\bigcup_{x \in \Gop} \supp(k_x)}$ the support of $P$, where $k_x$ denotes the 
Schwartz integral kernel of the operator $P_x$ in the family $P = (P_x)_{x \in \Gop}$. 
Then $\supp_{\mu}(P) := \mu(\supp(P))$ will be called \emph{reduced support}. 
If $\supp_{\mu}(P) \subset \G$ is a compact subset (with regard to the locally compact topology of $\G$), then we 
call $P$ \emph{uniformly supported}. 

\begin{Def}
Let $\G \rightrightarrows \Gop$ be a Lie groupoid.
We define by $\Psi_u^m(\G)$ the class of smooth families $P = (P_x)_{x \in \Gop}$ of pseudodifferential operators of order $m \in \Rr$ on the $s$-fibers $(\G_x)_{x \in \Gop}$ of $\G$ 
which are right invariant and uniformly supported.
\label{Def:PsDo}
\end{Def}

Denote by $S_{cl}^m(A^{\ast}) \subset C^{\infty}(A^{\ast})$ the H\"ormander space of classical (i.e. polyhomogenous) symbols on the Lie-algebroid $A = A(\G)$.
If we identify $U \cong \Omega \times s(U)$ via a fiber preserving diffeomorphism then elements $a \in S(A^{\ast})$
are families $(a_x)$ with $a_x \in S^m(T_x^{\ast} \G_x)$.
They are locally parametrized symbols $s(U) \mapsto S_{cl}^m(\Omega), x \mapsto a_x$. 

Given $P \in \Psi_c^m(\G)$ the principal symbol is given by
\[
\sigma_m(P)(\xi) = \sigma_m(P_x)(\xi), \xi \in A_x^{\ast}(\G) = T_x^{\ast} \G_x.
\]

[By homogeneity $\sigma_m(P) \in C^{\infty}(S^{\ast} \G)$ for the cosphere bundle of groupoids, but we don't need this fact.]

\textbf{Quantization:} Let us describe the quantization rule $S_{cl}^m(A^{\ast}) \ni a \mapsto \op_c(a) \in \Psi_c^m(\G)$.
Fix an open subset $\Gop \subset U \subset \G$ and a local diffeomorphism $\Phi \colon U \to \A(\G)$.
Here $\Phi$ maps diffeomorphically to an open neighborhood of the zero section in $\A(\G)$ and $d\Phi = \mathrm{id}$\footnote{This map is
given via the exponential map by first fixing an invariant connection $\nabla$ on $A(\G)$. See \cite{NWX} for more examples of
this kind.}.
Let $\chi \colon \G \to \Rr_{+}$ be a compactly supported, smooth cutoff function with support contained in $U$.

Set for $\xi \in A_x^{\ast}(\G), \gamma \in \G_x$
\[
e_{\xi}(\gamma) := \chi(\gamma) e^{i\scal{\Phi(\gamma)}{\xi}}. 
\]

For $a \in S_{cl}^m(A^{\ast})$ we set
\[
k(\gamma) := \int_{A_{r(\gamma)}^{\ast}} e_{-\xi}(\gamma) a(r(\gamma), \xi)\,d\xi.
\]

Then $k \in I_c^m(\G, \Gop)$ and $\op_c(a) \in \Psi_c^m(\G)$. We refer to \cite{V} for a proof of this.
By the representation for $\G$-operators, $P = \op_c(a)$ is given by
\begin{align}
(P f)(\gamma) &= \int_{\G_{s(\gamma)}} k(\gamma \eta^{-1}) f(\eta) \,d\mu_{s(\gamma)}(\eta) \notag \\
&= \int_{\G_{s(\gamma)}} \int_{A_{r(\gamma)}^{\ast}} \chi(\gamma \eta^{-1}) e^{i \scal{\Phi(\gamma \eta^{-1})}{\xi}} a(r(\gamma), \xi) f(\eta)\,d\xi \,d\mu_{s(\gamma)}(\eta). \tag{$Q$} \label{quant}
\end{align}

If $a_m$ denotes the homogenous principle symbol we have $\sigma_m(\op_c(a)) = a_m$ (c.f. \cite{V}).

\begin{Exa}
The reader should verify the following examples.

\begin{itemize}
\item The smooth vector bundle $\pi \colon \G = E \to M$ and $s = r = \pi$ yields
\[
(Pf)(v) = \int_{E_x} k_P(v - w) f(w) \,d\mu_x(w), x = \pi(v)
\]

as a $\G$-pseudodifferential operator. Here $\mu_x$ is obtained from a positive one density.

\item Given the smooth manifold $M$. The pair groupoid $\G = M \times M$ from example \ref{Exa:grpds} yields
\[
(P g)(z, x) = \int_{M} k_P(z, y) g(y,x) \,dy.
\]

\item $\G = \Gop = M$ for the manifold $M$. Then $P = M_{g}$ for $g \in C^{\infty}(M)$ acts by multiplication
$(P f)(x) = g(x) \cdot f(x)$. 
\end{itemize}

\label{Exa:Gops}
\end{Exa}

The first two examples combine in the case of the tangent groupoid below. 


\section{Orbit foliation and desingularization}

\label{desing}


The internal category of Lie groupoids is flexible enough to accomodate a large number of manifolds with singular and not necessarily integrable foliations.
Let $\M \in \Cinfty$, by a \emph{foliation} we mean a finitely generated submodule $\F$ of $\Gamma_c^{\infty}(T \M)$ that is stable under Lie bracket.
In this section we show that to a given Lie groupoid we obtain a smooth foliation by the orbits of the groupoid and vice versa the holonomy groupoid of a smooth
foliation can be constructed. We however will restrict ourselves to projective foliations of a manifold, which are foliations defined via a Lie algebroid.
We make further assumptions on the foliation which makes it \emph{almost regular}. 

The holonomy groupoid of a foliation provides a means to desingularize the space of leaves of a foliation, which is in general a badly behaved object. 
As a further objective we describe how a Lie groupoid can be viewed as a desingularization of a non-compact manifold with certain geometric
singularities. This second part is further examined in the final sections of this work where we consider so-called manifolds with a Lie structure at infinity. 

Denote by $\partial \M$ the stratified boundary of the manifold with corners $\M$ and by $\M_0 := \M \setminus \partial \M$ the interior. 
We fix the notation $\V_b := \{V \in \Gamma(T \M) : V \ \text{tangent to} \ \partial \M\}$ for the \emph{maximal Lie structure}, consisting of vector fields tangent to the boundary of $\M$. 

\begin{Def}
A \emph{singular manifold} is a tuple $(\M, \A)$ such that $\M \in \Cinfty$ and $\pi \colon \A \to M$ is a Lie algebroid
with anchor $\varrho \colon \A \to T\M$. The singular manifold $(\M, \A)$ is \emph{almost injective} if $\varrho(\A_x) \subset T_x \M_0$ is injective for each $x \in \M_0$
and \emph{sub-maximal}, if $\varrho_{\ast} \Gamma(\A) \subset \V_b$ is a Lie subalgebra of $\V_b$. 
\label{Def:singmf}
\end{Def}

We note that the anchor map $\varrho$ of a singular manifold $(\M, \A)$ induces a foliation of $\M$ which is a projective singular foliation of the manifold with corners. 
Next we recall the definition of a manifold with Lie structure at infinity following \cite{ALN}, called Lie manifold below.

\begin{Def}
A \emph{Lie manifold} is a tuple $(M, \V)$ where $M$ is a compact manifold with corners and $\V \subseteq \Gamma(TM)$ is a subspace
such that the following conditions hold.

\emph{i)} $\V$ is closed under Lie bracket. 

\emph{ii)} $C^{\infty}(M) \V = \V$ and $\V$ is a finitely generated  $C^{\infty}(M)$-module. 

\emph{iii)} $\Gamma_c(T M_0) \subset \V$. 

\emph{iv)} $\V \subset \V_b$. 

\label{Def:Liemf}
\end{Def}

The following Proposition is a consequence of the Serre-Swan theorem: projective modules over the ring of smooth functions
on a compact manifold are smooth vector bundles. 

\begin{Prop}
Given a compact singular manifold $(M, \A)$ with anchor map $\varrho$ that is almost injective and submaximal, then $\V := \varrho_{\ast}(\Gamma(\A))$ defines a Lie manifold $(M, \V)$.
On the other hand if $(M, \V)$ is a Lie manifold, then there exists a Lie algebroid with anchor $\varrho$ such that $\V = \varrho_{\ast}(\Gamma(\A))$. 
\label{Prop:Liemf}
\end{Prop}


\begin{Def}
Given a Lie groupoid $\G \rightrightarrows \Gop$ we denote by $\O_x := r(s^{-1}(x)) = s(r^{-1}(x))$ the \emph{orbit} in $x \in \Gop$.
The \emph{orbit space} of $\G$ is defined as $B \G := \Gop / \sim$.
\label{Def:orbit}
\end{Def}

\begin{Rem}
If $\G$ is a free and proper groupoid when acting on its space of units $\Gop$, then $\G$ is Morita equivalent to its space of orbits $B \G$.
This can be seen by the proof of Lemma \ref{Lem:Morita}.
In particular if two Lie groupoids $\G$ and $\H$ are Morita equivalent their spaces of orbits agree.
\label{Rem:orbits}
\end{Rem}

We define the \emph{holonomy groupoid} $\Hol(\F) \rightrightarrows M$ associated to a smooth foliation $\F \to M$ as being the smallest
$s$-connected Lie groupoid such that the orbits are the leaves of the foliation $\F$. 
By $s$-connected we mean a groupoid with connected $s$-fibers. The minimality means the following: Set $\G = \Hol(\F)$, then if $\H \rightrightarrows M$
is another $s$-connected Lie groupoid with orbits the leaves of $\F$, then there is a strict and surjective morphism $\H \to \G$. 
The naive approach to the construction of $\Hol(\F) \rightrightarrows M$ is to take the graph of the equivalence relation given by
\emph{being on the same leaf}. Unfortunately, this does not yield a smooth groupoid in general. Furthermore, in the case of more singular foliations the construction of the 
holonomy groupoid is even more tricky, since the rank of the foliation is no longer constant. For the foliations underlying a Lie manifold we refer to \cite{D} for the construction of a
suitable notion of holonomy groupoid. 


\subsection{Examples of Lie manifolds}

We recall a selection of particular examples of Lie manifolds. The table \ref{table:Lie} is a summary of structures that often recur in the literature.

\begin{Exa}[The $b$-structure]
Let $M$ be a compact manifold with boundary $\partial M \not= 0$ and let $M_0 := M \setminus \partial M$ denote the interior. Let $p$ be a fixed boundary defining function. The Lie structure $\V_b$ on $M$ is locally generated
by vector fields $\{x_1 \partial_{x_1}, \partial_{y_2}, \cdots, \partial_{y_n}\}$ in local coordinates where $x_1 = p$. 
Attaching a semi-infinite cylinder to $M_0$ we obtain the cylindrical end manifold $\M = M_0 \cup (\partial M \times (-\infty, 0]_t)$. The Riemannian metric on the 
far end of the cylinder takes the form $g^{cyl} = g_{\partial M} + dt^2$. The substitution (Kontradiev transform) $r = e^t$ yields
the metric $g = g_{\partial M} + (r^{-1} dr)^2$ in a collar neighborhood $\cong \partial M \times (0,1]_r$ of the boundary $\partial M$ in $M$. 
We obtain a Lie manifold $(M, \V_b)$ which models a manifold with cylindrical end. In the more general case of a manifold with corners
and Lie structure $\V_b$, we can view this as a model for a manifold with polycylindrical ends.  
Let $\A_{b} \to M$ the the Lie algebroid, i.e. the $b$-tangent bundle. The Lie algebroid structure of $\A_{b}$ consists of the Lie bracket which is induced by the standard bracket on $TM$ and the anchor which is the inclusion into the tangent bundle.
For a manifold with corners the integrating groupoid for $\A_b$ is given as follows, cf. \cite{MN2}. As a set 
\begin{align}
\Gamma_b(M) &= \{(x, y, \lambda_1, \cdots, \lambda_N) \in M \times M \times (\Rr_{+})^N : p_i(x) = \lambda_i p_i(y), \ \forall i \in I\}. \label{bgrpd}
\end{align}
The structural maps are $s(x, y, \lambda) = y, \ r(x, y, \lambda) = x$, multiplication is given by $(x, y, \lambda) \circ (y, z, \mu) = (x, z, \lambda \cdot \mu)$ and inverse $(x, y, \lambda)^{-1} = (y, x, \lambda^{-1})$ with entrywise vector multiplication. 
Note that as a set $\Gamma_b(M)$ is 
\[
M_0 \times M_0 \cup \bigcup_{i \in I} F_i^2 \times (\Rr_{+})^{\codim(F_i)}
\]
The topology is defined by $M_0 \times M_0 \ni (x_n, y_n)$ converges towards $(x, y, \lambda) \in \Gamma_b(M)$ if and only if
$x_n \to x, \ y_n \to y$ and $\frac{p_i(y_n)}{p_i(x_n)} \to \lambda_i, \ i \in I, \ n \to \infty$. 
The \emph{$b$-groupoid} is defined as the union of connected components of the $s$-fibers, i.e. the $s$-connected component of $\Gamma_b(M)$. 
\label{Exa:b}
\end{Exa}

\begin{Exa}[The $sc$-structure]
A scattering Lie manifold is a compact manifold $M$ with boundary $\partial M \not= \emptyset$ endowed with the Lie structure of scattering vector fields, i.e. the module of vector fields $\V_{sc} = p \V_b$ where $p$ is the boundary defining function. 
Let $\A_{sc} \to M$ denote the Lie algebroid, i.e. the scattering tangent bundle. In local coordinates where $x_1 = p$ a local basis of $\A_{sc}$ can be chosen as $\{x_1^2 \partial_{x_1}, x_1 \partial_{x_j}\}$, $j > 1$.
The Lie algebroid structure of $\A_{sc}$ consists of the Lie bracket which is induced by the standard bracket on $TM$ and the anchor which is given by the inclusion into the $b$-tangent bundle composed with the inclusion into the tangent bundle.  
An integrating groupoid for $\A_{sc}$ is given as a set by
\[
\G_{sc} = T_{\partial M} M\cup (M_0 \times M_0) \rightrightarrows M.
\]

Here the tangent bundle restricted to $\partial M$ is a viewed Lie groupoid which is glued to the pair groupoid on the interior.
If $M$ is a compact manifold with corners the scattering groupoid takes the form
\[
\G_{sc} = \left(\bigcup_{F \in \F_1(M)} T_{F} M \right) \cup (M_0 \times M_0) \rightrightarrows M.
\]

\label{Exa:sc}
\end{Exa}

\begin{table}[H]
\caption{Singular manifolds}
\centering
\begin{tabular}{c c}
\hline \hline
Local generators & Lie structure \\ [0.5ex]
\hline
$\partial_{x_i}$ & smooth, compact \\
$x \partial_x, \ x \partial_{y_i}$ & asymptotically euclidean \\
$x \partial_x, \ \partial_{y_i}$ & $b$-type \\
$x^n \partial_x, \ \partial_{y_i}$ & general cusps ($n \geq 2$) \\
$x^2 \partial_x, \ x \partial_{y_i}$ & scattering \\
$x^l \partial_x, \ x^l \partial_{y_i}, \ \partial_{z_j}$ & edge ($l$-fold) \\
$x^2 \partial_x, \ x \partial_{y_i}, \ \partial_{z_j}$ & fibered cusp \\
etc. & \\ [1ex]
\hline
\end{tabular}
\label{table:Lie}
\end{table}

\subsection{Remarks on the generalized index}

Consider the space of leaves of the foliation $M / \F$ which is in general a highly singular space.
In order to obtain information about the leaf space Alain Connes introduced new geometric objects which capture the transverse
structure of a foliation (see also table \ref{table:invs}).
The main object is the holonomy Lie groupoid $\Hol(\F) = \G \rightrightarrows M$ which is viewed as a desingularization of the leaf space $M / \F$.
The analytic $K$-theory is defined as the $K$-theory of the $C^{\ast}$-algebra $C_r^{\ast}(\G)$.
The geometric realization of the nerve $\NG$ of the groupoid is the classifying space $\BG$ of the groupoid $\G$.
The generalized $\G$-equivariant $K$-theory $K_{\ast, \tau}(\BG)$ due to Baum and Connes (cf. \cite{C}) is the topological $K$-theory
associated to the foliation. The Baum-Connes conjecture for foliations proposes a link between these two types of $K$-theory (this is a special case of 
the geometric Baum-Connes conjecture and we refer to Section \ref{geoBC}). In index theory of foliations one is interested in tangentially elliptic differential operators, given by families $P = (P_x)_{x \in M}$. 
The problem which occurs is that if $P_x$ is a differential operator on a leaf $L_x$ of the smooth foliation $\F$ then
it may have infinite dimensional kernel as leafs may be non-compact manifolds.
Connes observed that the index of such an operator makes sense if viewed as a generalized index with values in the $K$-theory of the $C^{\ast}$-algebra
of the desingularization groupoid $\G$. For Lie groupoids we define the generalized analytic index in Section \ref{Singmf} via the adiabatic deformation \cite{MP}. 

\begin{table}[H]
\caption{Functorial invariants (examples)}
\centering
\begin{tabular}{c c}
\hline \hline
$B \G$ & orbit space $\Gop / \sim$ \\
$\IG$ & isotropy groupoid $\bigcup_{x \in M} \G_x^x$ \\
$\NG$ & nerve of the category $\G = (\Gop, \Gmor)$ \\
$\BG$ & classifying space for proper $\G$-actions \\  [1ex]
\hline
\end{tabular}
\label{table:invs}
\end{table}

\subsection{Analysis on singular manifolds}

We can use the idea of desingularization also in analysis. This is particularly relevant when we study index theory and analysis
on manifolds which are non-compact. 
If we have a definition of the form of differential operators on a certain class of non-compact manifolds, then an obvious
important question is when such an operator is Fredholm and to study different notions of ellipticity.
Consider a non-compact manifold $M_0$ and view $M$ as a compactification which is endowed with a smooth structure. 
We can do this by viewing a Lie groupoid $\G \rightrightarrows M$ as a desingularization of geometric singularities of $M_0$ in the following sense.
The connected components of the orbits of $\G$ form the stratified boundary $\partial M$ of $M$ and $M_0$ identifies with the
interior $M \setminus \partial M$. The orbit foliation $\F \to M$ associated to $\G$ is required to be \emph{almost regular}, i.e. there is a maximal stratum or leaf of constant dimension.
We make a stronger assumption on the strata, i.e. they should be embedded, codimension one submanifolds of $M$.
In this way we can obtain a compact manifold with corners $M$. On a non-compact manifold $M_0$ it is no longer true that a differential operator which is elliptic (i.e. the principal symbol is pointwise invertible)
is automatically Fredholm, cf. \cite{MN2}. Fix a Lie groupoid $\G \rightrightarrows M$ and consider the generalized analytic index $\ind_a \colon K^{0}(\A^{\ast}(\G)) \to K^{0}(C_r^{\ast}(\G))$, cf. Section \ref{Singmf}. 
Let $D$ be a differential operator or pseudodifferential operator on $M_0$ which is compatible with the singular structure of $M_0$, in particular the principal symbol is invariantly defined on $\A^{\ast}$, see \cite{MN2} for a precise definition.
Assume that $D$ is elliptic, i.e. the principal symbol is invertible on $\A^{\ast}$. 
As observed in \cite{MN2} vanishing of the generalized analytic index yields a lower order perturbation for given elliptic operators on
manifolds with polycylindrical ends. In this note we consider some additional data, namely let $Y \subset M$ be an embedded codimension one submanifold which is transversal in the sense that
for each $y \in Y$ we have $T_y M = T_y Y + \varrho_y(\A_y)$. 
This is an appropriate setup for boundary value problems on non-compact manifolds with geometric singularities.
We then construct in Section \ref{Singmf} a generalized analytic index $\widetilde{\ind_a}$ for boundary value problems and 
generalize results from \cite{MN2} to this setting.
A number of open questions about the invariant $\ind_a^{+}$ remain unresolved and are left for a later work.
Among these questions is the study of boundary value problems (e.g. classical boundary value problems like the Dirichlet problem) in connection
with the vanishing of the index, similarly to the analysis in \cite{MN2}.

\section{Deformation groupoid}


In the following we study the base case to our previous definition: A smooth compact manifold $M$ with empty boundary, i.e. $\partial M = \emptyset$ and the Lie structure $\V = \Gamma(TM)$ of all vector fields. 
We will return to general Lie manifolds, as well as their index theory, in Section \ref{Singmf}. 
The basic idea underlying the groupoid proof of the Atiyah-Singer index theorem is a deformation construction of the $C^{\ast}$-algebras of compact operators $\K := \K(L^2(M))$
into $C_0(T^{\ast} M)$. This is accomplished by noting that the $C^{\ast}$-algebra of the pair groupoid $M \times M$ yields $\K$
and the $C^{\ast}$-algebra of the tangent bundle (viewed as a groupoid) $TM$ is $C_0(T^{\ast} M)$.

We introduce now the tangent groupoid and define its smooth structure.

Let $M$ be a closed, compact manifold. Define the tangent groupoid 
\[
\G = \G_M^t := (M \times M \times ]0,1]) \cup \TM, \ \Gop = M \times [0,1].
\]

The structural maps are described as follows with $O_M$ denoting the zero-section
\begin{align*}
& (x, t) \mapsto (x, x, t) \in M \times M \times ]0,1], \ x \in M, \ t > 0, \\
& (x, 0) \mapsto x \in O_M = M \subset \TM, \ t = 0, \\
& r(x, y, t) = (x, t), \ x \in M, \ t > 0, \\
& r(x, v) = (x, 0), \ v \in T_x M, \ x \in M, \\
& s(x,y,t) = (y, t), \ y \in M, \ t > 0, \\
& s(y, w) = (y, 0), \ y \in M, \ w \in T_y M.
\end{align*}

The composition is defined by
\begin{align*}
& (x, y, t) \circ (x, z, t) = (x, z, t), \ t > 0, \ x,y,z \in M, \\
& (x, v) \circ (x,w) = (x, v + w), \ v, w \in T_x M.
\end{align*}
 
Since $(T_x M, +)$ are additive groups we obtain a groupoid $\G_2 := \TM = \bigcup_x T_x M$.
Secondly, $\G_1 := (M \times M) \times (0,1]$ is a groupoid as a product of the pair groupoid and $(0,1]$ viewed
as the trivial set groupoid.

Therefore, algebraically $\G$ is a groupoid. We will clarify the topology of the groupoid and show that is indeed a Lie groupoid. 

\begin{Not}
Denote by $j \colon N \hookrightarrow M$ an inclusion of smooth manifolds. Let $j_{\ast} \colon TN \to TM$ denote the pushforward of the inclusion which is one to one and by $j^{\ast} \colon TM \to TM_{|N}$ the pullback which is onto. 
Then we set $\N^j = j^{\ast} TM / j_{\ast} TN$ for the corresponding normal bundle. 
\label{Not:normal}
\end{Not}

\begin{Thm}
Given a $C^{\infty}$-inclusion $j \colon N \hookrightarrow M$, then there is $\Phi$ depending on $N \subset U \subset M$ an open neighborhood, such that $\Phi_{|O_N} = j_{N \hookrightarrow M}$
and $\Phi$ is a local diffeomorphism, i.e. there is an open neighborhood $O_N \subset V \subset \N^j$ such that $\Phi \colon U \iso V$ is a diffeomorphism. 
Hence we have a commuting diagram 
\[
\xymatrix{
\N^j \ar[r]^{\Phi} & M \\
V \ar@{^{(}->}[u] & \ar@{>->>}[l] U \ar@{^{(}->}[u] \\
O_N \ar@{^{(}->}[u] \ar@{>->>}[r] & N \ar@{^{(}->}[u] 
}
\]

\label{Thm:normfib}
\end{Thm}

\begin{proof}
We refer to \cite{S}, Thm. 4.1.1. 
\end{proof}

\begin{Lem}
Let $M$ be a smooth manifold and let $\Delta \colon M \hookrightarrow M \times M$ denote the diagnoal inclusion.
Then we have a non-canonical isomorphism $\N^{\Delta} \cong TM$. 
\label{Lem:diag}
\end{Lem}

\begin{proof}
We have the isomorphism $T_{(x,x)}(M \times M) \cong T_x M \oplus T_x M$ and $(\Delta^{\ast} T(M \times M))_{(x,x)} \cong T_x M \oplus T_x M$ as well as $(\Delta_{\ast} TM)_{(x,x)} \cong \Delta_{T_x M} \subset T_x M \times T_x M$. This implies by definition of the normal bundle that $\N^{\Delta} \cong TM$. 
\end{proof}

\begin{Thm}
The tangent groupoid $\G_M^t \rightrightarrows M$ is a Lie groupoid.
\label{Thm:Liegrpd}
\end{Thm}

The following exposition relies strongly on the paper \cite{CGF}. The idea of the proof (via \emph{transport of structure}) can also be adapted to the construction of the adiabatic groupoid, cf. \cite{LR} and the remark after the proof.

\begin{proof}

\emph{1. Smooth structure:} Let $t_0 > 0$ and consider the manifold $TM \times [0, t_0]$, an open set $U$ such that $TM \times \{0\} \subset U \subset TM \times [0, t_0]$.
Define $\Phi \colon U \to \G_M$ by 
\begin{align*}
\Phi(x, v, t) &:= (\exp_x(tv \frac{1}{2}), \exp_x(-\frac{1}{2} tv), t), \ t > 0, \\
\Phi(x, v, 0) &:= (x, v), \ t= 0.  
\end{align*}

For $t$ fixed, $\exp$ yields a diffeomorphism of an open neighborhood of the zero section $O_M \subset V \subset TM$ onto an open neighborhood
of the diagonal $\Delta \subset U \subset M \times M$. 
Choose $U$ (via the tubular neighborhood theorem) such that $\Phi$ is bijective. 
Then $U$ is an open submanifold with boundary, hence $TM = \partial \G_M$ is the boundary, the diagonal is identified with $M \times [0,1]$ and the interior $\mathring{\G}_M = M \times M \times (0,1]$. 

\emph{2. Structural maps:} Consider the unit inclusion $u \colon \Gop \hookrightarrow \G_M$, then $u_{|M \times (0,1]}$ is $C^{\infty}$ on its domain. 
The maps $\Phi^{-1} \circ u(x,t) = \Phi^{-1}(x, x, t) = (x, 0, t), \ t > 0$ and $\Phi^{-1} \circ u(x, 0) = \Phi^{-1}(x, 0, 0) = (x, 0, 0), \ t = 0$ are smooth on their domain of definition.
Hence restricting $u$ to $u^{-1}(\Phi(U))$ and application of $\Phi^{-1}$ gives the smoothness of $u$. 

\begin{align*}
\xymatrix{
M \times I = \G_M^{(0)} \ar@{^{(}->}[r]^-{u} & \G_M \\
& TM \times I \ar[u]^{\Phi} 
}
\end{align*}

We argue similarly for the range map $r \colon\G_M \to \G_M^{(0)}$. Note that $r_{|M \times M \times (0,1]}$ is a submersion. 
We have 
\[
r \circ \Phi(x, v, t) = r(\exp_x(\frac{1}{2} tv), \exp_x(-\frac{1}{2} v t)) = (\exp_x(\frac{1}{2} tv), t), \ t > 0
\]

and
\[
r \circ \Phi(x, v, 0) = (x,0), \ t =0. 
\]

Compose $\Phi$ with $r_{|\Phi(U)}$

\begin{align*}
\xymatrix{
M \times I = \G_M^{(0)} & \ar[l]^-{r} \G_M \\
& TM \times I \ar[u]^{\Phi}
}
\end{align*}

then this is a $C^{\infty}$-map with maximal rank, hence $r$ is a submersion. From the smoothness of the inversion, we obtain that $s$ is a submersion via $s = r \circ i$.
\end{proof}


\begin{Rem}
\emph{i)} Generalizing the construction of the tangent groupoid we fix the data $j \colon N \hookrightarrow M$ an inclusion of smooth manifolds.
Then we define the deformation to the normal cone $\D(M, N) = \N^j \times \{0\} \cup M \cup (0,1]$. The topology of the normal cone deformation
is defined using the general normal fibration, Theorem \ref{Thm:normfib}. 
Via Lemma \ref{Lem:diag} we find that $\D(M \times M, M) = \G_M^t$ and we recover the locally compact topology of $\G_M^t$.
The gluing does not depend on the particular local diffeomorphism used in the proof of Theorem \ref{Thm:Liegrpd}.  

\emph{ii)}
 Given a Lie algebroid, the \emph{adiabatic groupoid} is defined as
\[
\Gad := \G \times (0,1] \cup \A(\G) \times \{0\}. 
\]

The groupoid $\Gad \rightrightarrows \Gop \times I$ is a Lie groupoid where the topology is defined by the glueing
\[
[0,1] \times \A(\G) \supseteq \O \ni (x, v, t) \mapsto \begin{cases} (0, v), & t= 0 \\
(t, \Exp(tv)), & t > 0 \end{cases}
\]

Here $\Exp \colon \A(\G) \to \G$ is a generalized exponential map as defined in \cite{LR}. 
See also in loc. cit. the proof of a generalized tubular neighborhood theorem. 
This is a special case of \emph{i)} in that we have the unit inclusion $u \colon \Gop \hookrightarrow \G$ (an inclusion of smooth manifolds) and
$\Gad = \D(\G, \Gop)$. 
\label{Rem:blowup}
\end{Rem}

\begin{Prop}
The following sequence 
\[
\xymatrix{
C_0(0,1] \otimes \K \ar@{>->}[r] & C^{\ast}(\G_M) \ar@{->>}[r]^{\sigma} & C_0(T^{\ast} M)
}
\]

is exact.

We obtain the isomorphisms in $K$-theory
\[
K_j(C^{\ast}(\G_M)) \iso K_j(C_0(T^{\ast} M)) = K^j(T^{\ast} M), \ j= 0,1.
\]

\label{Prop:SES}
\end{Prop}

\begin{proof}
By the previous result $\G = \G_1 \cup \G_2$ with $\G$ a smooth groupoid and $\G_1$ open, $\G_2$ closed.
Then let $e_0$ be the homomorphism induced by restriction $C_c^{\infty}(\G) \to C_c^{\infty}(\G_2)$.
Clearly, $\ker e_0 \cong C^{\ast}(\G_1)$ and we have the exact sequence
\[
\xymatrix{
C^{\ast}(\G_1) \ar@{>->}[r] & C^{\ast}(\G) \ar@{->>}[r]^{e_0} & C^{\ast}(\G_2).
}
\]

Since $C^{\ast}(M \times M) \cong \K(L^2(M))$ we obtain
\[
C^{\ast}(\G_1) = \overline{C_c^{\infty}(0,1] \otimes C_c^{\infty}(M \times M)} \cong C_0(0,1] \otimes \K.
\]

This latter algebra is contractible: Set $\alpha_t(f \otimes K) := f(t \cdot) \otimes K, t \in [0,1], f \in C_0(0,1]$
which gives a contracting homotopy.
Hence $K_j(C_0(0,1] \otimes \K) \cong 0$. Also $C^{\ast}(\G_2) = C^{\ast}(\TM) \cong C_0(\TM)$ via Fourier transform.
The six-term exact sequence in K-theory then yields the isomorphisms.  
\end{proof}

\section{Connes Thom isomorphism}


Let $\G \rightrightarrows M$ be a Lie groupoid and $h \colon (\G, \cdot) \to (\Rr^N, +)$ a homomorphism.
Then $h$ induces a natural right action of $\G$ on $M \times \Rr^N$.
Setting $B\G = Z / \sim$ the orbit space and $\BG = Z = M \times \Rr^N$ the corresponding classifying space.

This defines the action groupoid $\G_h = Z \rtimes \G$ with the following structure 
\begin{itemize}
\item $s(\gamma, v) = (s(\gamma), v + h(\gamma))$,

\item $r(\gamma, v) = (r(\gamma), v)$,

\item $(\gamma, v) \cdot (\eta, v+ h(\gamma)) = (\gamma \cdot \eta, v)$, 

\item $u(x, v) = (x, v)$ (note that $h(x) = 0$, since $h$ is a homomorphism), 

\item $(\gamma, v)^{-1} = (\gamma^{-1}, v + h(x))$ (note that $h(\gamma) + h(\gamma^{-1}) = 0$). 
\end{itemize}

\begin{Thm}[Elliot-Natsume-Nest]
Given data as above: $(\G, h)$ an amenable Lie groupoid with a homomorphism $h$ to $\Rr^N$ and assume that $N$ is even.
Then there is a map in $K$-theory 
\[
\CT_h \colon K_j(C^{\ast}(\G)) \to K_j(C^{\ast}(\G_h))
\]

such that 

\emph{i)} $\CT_h$ is natural, 

\emph{ii)} $\G = \Gop$ yields $\CT_h = \beta$, the Bott map, 

\emph{iii)} $\G = \G_M^t$ the tangent groupoid, then $\CT_h^{t=0} = \tau$ yields the Thom isomorphism. 

\label{Thm:ENN}
\end{Thm}

\begin{proof}
The proof is contained in \cite{ENN} or alternatively \cite{FS}. The present formulation of the theorem is from \cite{CLM}. 
\end{proof}

Let $\G = \G_M^t$ and define $h^t \colon (\G, \cdot) \to (\Rr^n, +)$ as follows
\begin{align*}
h^t(x, y, t) &= \frac{j(x) - j(y)}{t}, \ t > 0, \\
h^t(x, v) &= dj_x(v), \ t = 0. 
\end{align*}

\begin{Lem}
Given $(\G = \G_M^t, h = h^t)$ then $h$ induces a free and proper action of $\G$ on $\Rr^N \times \Gop$. 
\label{Lem:ENN}
\end{Lem}

\begin{proof}
We show that the semi-direct product groupoid $\G \rtimes_h Z$ has quasi-compact isotropy (i.e. stabilizers of the action) and $(r, s) \colon \G \rtimes Z \rightrightarrows Z$ is a closed map. 
This will imply that the action is free and proper by Proposition \ref{Prop:proper}. 
Note that $h = h^t$ is injective, hence $h(\gamma) = 0$ if and only if $\gamma$ is a unit.  
We obtain that 
\[
(\G_h)_z^z = (\G_h)_{v,x}^{v,x} = \{(\gamma, v) \in \G \times \Rr^N : \gamma \in \G_x^x, \ h(\gamma) = 0\} = \{(x, v)\}. 
\]

Hence the isotropy is trivial, in particular quasi-compact.
We are left to show the closedness of the map $r \oplus s$. 
Let $(\gamma_n, v_n) \in \G \times \Rr^N$ be a sequence such that 
\[
\lim_{n \to \infty} (r \oplus s)(\gamma_n, v_n) = (z, w)
\]

for $(z, w) \in Z^2$. 
We need to find a subsequence of $(\gamma_n, v_n)_n$ which converges to a preimage $(\gamma, v)$ of $(z, w)$ with regard to $r \oplus s$. 
Consider the cases $(\gamma, v) = (x, y, t, v) \in \G_1$ and $(\gamma, v) = (x, \tilde{v}) \in \G_2$. 
In the first case there is by compactness of $M$ a subsequence converging to $(x, y, t, v)$. 
For the second case we have $s(\gamma) = r(\gamma)$ and obtain a subsequence converging in the fiber of $TM$. 
\end{proof}

\begin{Lem}
The orbit space $B \G_h = Z / \sim$ of the Lie groupoid $\G_h = Z \rtimes_h \G_M^t$ is diffeomorphic to the deformation space $\D(\Rr^N, M) = \Rr^N \times (0, 1] \cup N(j) \times \{0\}$. 
\label{Lem:orbit}
\end{Lem}

\begin{proof}
Let $B\G_h = B \G_{1,h} \cup B \G_{2,h}$, where $\G_1 = M \times M \times (0,1]$ and $\G_2 = TM \times \{0\} \cong TM$. 
Furthermore, $B \G_{1,h}$ is the quotient $\G_1^{(0)} \times \Rr^N / \sim$ by the action of $\G_1$, $B \G_{2,h}$ is the quotient 
$\G_2^{(0)} \times \Rr^N / \sim$ by the action of $\G_2$. 
We have two isomorphisms $\varphi_1 \colon \Rr^N \times (0,1] \to M \times (0,1] \times \Rr^N / \sim = B \G_{1,h}$ and $\varphi_2 \colon N(j) \times \{0\} \to M \times \Rr^N / \sim$. 
To describe $\varphi_1$ we choose a basepoint $x_0 \in M$ and write $\varphi_1(t, v) = ((x_0, t), v)$. 
We endow $B \G_h$ with the locally compact topology of the deformation space $D(\Rr^N, M)$ as described in Remark \ref{Rem:blowup}. 
\end{proof}

\begin{Lem}
There is a Morita equivalence $B\G_h \sim_{\M} Z \rtimes_h \G$, hence there is an isomorphism in $K$-theory
\[
K_j(C^{\ast}(Z \rtimes_h \G)) \cong K^j(B \G_h), \ j = 0,1. 
\]
\label{Lem:Morita}
\end{Lem}

\begin{proof}

We view the orbit space $B \G_h$ as a trivial groupoid with units $B \G_h$ and differentiable structure given in the deformation construction. Then the Morita equivalence is written in terms of the groupoid actions (note that $Z = \mathbb{B} \G_h$, the \emph{classifying space})

\[
\begin{tikzcd}[every label/.append style={swap}]
B \G_h \ar[d, shift left] \ar[d] \ar[symbol=\circlearrowleft]{r} & \ar{dl}{p} \mathbb{B} \G_h \ar{dr}{q} & \ar[symbol=\circlearrowright]{l} \G_h \ar[d, shift left] \ar[d] \\
B \G_h & & \mathbb{B} \G_h
\end{tikzcd}
\]

Here $p$ is the quotient map $\mathbb{B} \G_h \to \mathbb{B} \G_h / \sim$ and $q$ is induced from the range map of the groupoid $\G_h$.
The action of $\G_h$ on $\mathbb{B} \G_h$ is free and proper and the action of $B \G_h$ is free and proper.
By Lemma \ref{Lem:orbit} we obtain that $\mathbb{B} \G_h / \sim \to \mathbb{B} \G_h$ is diffeomorphic to $\mathbb{B} \G_h$ and by definition $B \G_h \setminus \mathbb{B} \G_h$ is diffeomorphic to $\G_h$. 
The assertion about $K$-theory follows by noting that we obtain a strong Morita equivalence of $C^{\ast}$-algebras $C^{\ast}(\G_h) \simM C^{\ast}(B \G_h)$ by Theorem \ref{Thm:functoriality} which in turn
furnishes the isomorphism in $K$-theory.
\end{proof}

\section{Proof of the index theorem}

\label{proof}

In this secton we fix a smooth and compact manifold $M$ without boundary, i.e. $\partial M = \emptyset$. 

\subsubsection*{The topological index} Fix a smooth embedding $j \colon M \hookrightarrow \Rr^N$ for $N$ a (sufficiently large) natural number. 
Denote by $dj \colon T^{\ast} M \hookrightarrow \Rr^{2N}$ the induced differential and let $N(j)$ be the normal bundle to the embedding of $M$ into $\Rr^N$. 
Fix $\tau_{N(j)} \colon K^{0}(N(j)) \to K^0(T M)$ the Thom-isomorphism (as commonly done in the definition of the topological index we identify $T^{\ast} M \cong TM$).

\begin{Def}
The \emph{topological index} is the pushforward $\ind_t = (dj)_{!} \colon K^{0}(T^{\ast} M) \to \Zz$ such that the following diagram commutes
\[
\xymatrix{
K^0(T^{\ast} M) \ar[d]_{\tau_{N(j)}^{-1}} \ar[r]^{\ind_t} & \Zz \ar[d]_{\beta} \\
K^0(N(j)) \ar[r]^{\psi} & K^0(\Rr^N).
}
\]

\label{Def:indt}
\end{Def}


The next step is easy: Rewrite the topological index in terms of deformation maps. We use the generalized Thom isomorphism which we described in the previous section.
 
\begin{Prop}
Apply the Connes-Thom map to the tangent groupoid $\G = \G_M^t$. We obtain a commuting diagram
\[
\xymatrix{
K^0(T^{\ast} M) \ar[d]_{(\CT_h^{t=0})^{-1}} \ar[r]^{\ind_t} & \Zz \\
K^0(N(j)) \ar[r]^{\psi} & K^0(\Rr^{2N}) \ar[u]^{\CT_h = \beta}. 
}
\]
\label{Prop:indt}
\end{Prop}

\begin{proof}
Set $\G = \G_M^t$ and fix a Haar system $(\mu_x)_{x \in \Gop}$ on $\G$. We obtain a Haar system $(\mu_z^h)_{z \in Z}$ on $\G_h =  Z \rtimes_h \G$ induced by $h$. 
Note that $\Rr^N$ acts on $C^{\ast}(\G)$ by automorphisms, i.e. for each $\chi \in \widehat{\Rr^N}$, a character, there is an automorphism $\alpha$ such that
$\Rr^N \ni v \mapsto \alpha_v(\cdot) \colon C_c^{\infty}(\G) \to C_c^{\infty}(\G)$. 
Set for example $\alpha_{v}(f) = e^{i \scal{v}{h(\gamma)}} f(\gamma), \ \forall \ f \in C_c^{\infty}(\G)$. 
It follows that $C^{\ast}(\G_h) \cong C^{\ast}(\G) \rtimes_{\alpha} \Rr^N$, \cite{CMR}. 
The Connes-Thom morphism for $\G = \G_M^t$ yields the isomorphisms in $K$-theory, Theorem \ref{Thm:ENN}
\begin{align*}
\CT_h \colon K_j(C^{\ast}(\G)) \iso K_j(C^{\ast}(Z \rtimes_h \G)) \cong K_j(C^{\ast}(\G) \rtimes \Rr^N). 
\end{align*}

By the Lemma \ref{Lem:Morita} we have $K_j(C^{\ast}(Z \rtimes_h \G)) \cong K^j(B \G)$ for $j = 0,1$. 
Inserting this into the Connes-Thom isomorphism above and evaluating for $t = 0$ on both sides yields 
\[
\CT_h^{t=0} \colon K_j(C^{\ast}(\G_M^t)) = K^j(TM) \iso K^j(B \G_h^{t=0}) \cong K^j(N(j)). 
\]

By the uniqueness result of Elliot-Natsume-Nest this recovers the Thom isomorphism, hence $\tau = \CT_{h}^{t=0}$. 
We therefore recover the definition of $\ind_t$ using the gernalized Connes-Thom map, applied to the tangent groupoid. 
\end{proof}

\subsubsection*{The analytic index}

\begin{Def}
Define the \emph{generalized analytic index} as the Kasparov product $\ind_a := - \otimes \partial_M$. In other words $\ind_a$ is defined such that the following diagram commutes
\[
\xymatrix{
K_0(C_0(T^{\ast} M)) \ar[d]_{\ind_a} \ar[r] & K_0(C^{\ast}(\G_M)) \ar[dl]_{(e_1)_{\ast}} \\
K_0(\K) \cong \Zz
}
\]

\label{Def:inda}
\end{Def}

\begin{Prop}
The generalized analytic index equals the Fredholm index, i.e. 
\[
\ind_a = \ind_a^{\F}.
\]

\label{Prop:inda}
\end{Prop}

\begin{proof}
We make use of the pseudodifferential calculus on groupoids following \cite{DL}. 
First the Lie algebroid is given by $\TM \times T[0,1] = A(\G_M) \to M \times [0,1]$ with anchor $\varrho_M \colon (x, v, t) \in \TM \times [0,1] \mapsto (x, t v, t, 0) \in \TM \times T[0,1]$.

Hence pseudodifferential operators are $t$-scaled operators, families parametrized by $[0,1]$ and obtained by the quantization \eqref{quant}.

Given $a \in S_{cl}^0(A^{\ast})$ elliptic with $P = \op_c(a)$. Then $P = (P_{t})_{t \in [0,1]}$ with
\[
P_{t} u(x,y,t) = \int_M \int_{T_x^{\ast} M} e^{\frac{\exp_x^{-1}(z)}{t} \cdot \xi} a(x, \xi) u(z, y) \frac{dz d\xi}{t^n} 
\]

for $t > 0$ and
\[
P_0 u(x, v,0) = \int_{T_x M} \int_{T_x^{\ast} M} e^{(x - w)\cdot \xi} a(x, \xi) u(x, w) \,dw d\xi
\]
for $t = 0$.

Note that $P_1$ is an ordinary pseudodifferential operator on $M$ and $\sigma_0(P_1) = a_0$. 

With $[a] \in K_0(C^{\ast}(\TM)) \cong K_0(C_0(T^{\ast} M))$ and $[P] \in K_0(C^{\ast}(\G_M))$ we therefore obtain
\[
[a] \otimes [e_0]^{-1} \otimes [e_1] = [P_1] \in K_0(\K).
\]

We note that $[P_1]$ equals $\ind(P_1)$ under the isomorphism $K_0(\K) \cong \Zz$. Since every class in $K_0(C_0(T^{\ast} M))$ comes
from an elliptic $0$-order symbol the proof is finished.
\end{proof}

\subsubsection*{The index theorem}

\begin{Thm}[Atiyah-Singer 1968]
Let $M$ be a smooth compact manifold without boundary. Then we have the equality of indices
\[
\ind_t = \ind_a. 
\]
\label{Thm:AS}
\end{Thm}

\begin{proof}
We make use of the naturality of the general $\CT$-functor. Denote by $\M$ the isomorphism in $K$-theory which is implemented by the Morita equivalence $M \times M \simM pt$. 
The following diagram commutes

\[
\xymatrix{
\Zz \ar@{==}[r] & \Zz \\
K_0(C^{\ast}(M \times M)) \ar@{>->>}[u]^{\M} \ar@{>->>}[r]^-{\CT_h^{t=0}} & K^0(\Rr^{2N}) \ar@{>->>}[u]^{\beta} \\
K_0(C^{\ast}(\G_M^t)) \ar@{>->>}[d]_{(e_0)_{\ast}} \ar[u]^{(e_1)_{\ast}} \ar@{>->>}[r]^-{\CT_h} & K_0(C^{\ast}(\G_h)) \ar[u]^{(e_1^h)_{\ast}} \ar@{>->>}[d]_{(e_0^h)_{\ast}} \\
K^0(T^{\ast} M) \ar@{>->>}[r]^-{\CT_h^{t=0}} & K^0(\N(j))
}
\]

We obtain from the previous diagram 

\begin{align*}
\ind_t &= \beta \circ \psi \circ \tau_{\N(j)}^{-1} = \M \circ (\CT_h^{t=0})^{-1} \circ (e_1^h)_{\ast} \circ (e_0^h)_{\ast}^{-1} \circ \CT_h^{t=0} \\
&= \M \circ (\CT_h^{t=0})^{-1} \circ (\CT_h^{t=0}) \circ (e_1)_{\ast} \circ \CT_h^{-1} \circ \CT_h \circ (e_0)_{\ast}^{-1} \circ (\CT_h^{t=0})^{-1} \circ \CT_h^{t=0} \\
&= \M \circ (e_1)_{\ast} \circ (e_0)_{\ast}^{-1} = \ind_a.
\end{align*}

This ends the proof of Atiyah-Singer. 
\end{proof}

\section{The Baum-Connes conjecture}

We give a brief overview of the (generalized) Baum-Connes conjecture for Lie groupoids. We will also explain in what sense the Baum-Connes conjecture stated for Lie groupoids can be viewed as 
a generalization of Atiyah-Singer index theory. We will see that the name Baum-Connes \emph{condition} is much more appropriate in the case of Lie groupoids.
Additionally, we recall the definition of two models for $K$-homology of Lie groupoids, the geometric $K$-homology due to Baum and Connes as well as the analytic (Baum-Douglas type) $K$-homology. 
A note of caution is in order here, since the geometric cycles as defined in \cite{C} are not what one usually expects: namely cycles which consist of geometric data with an equivalence relation stable
with regard to the usual constructions, like vector bundle modification and bordism. To the authors knowledge such a formulation of geometric $K$-homology for Lie groupoids has not yet appeared in complete form in the literature.
The definitions given here are in any case sufficient for our immediate purpose in this work.

\subsection{Preparations}


\subsubsection{$KK$ and $E$-theory} Kasparov's $KK$-theory (see e.g. \cite{K}) should be studied from at least two equally important viewpoints. On the one hand
$KK \colon C^{\ast} \to \Ab$ is a bifunctor which yields a generalization of $K$-theory and $K$-homology of $C^{\ast}$-algebras.
On the other hand $KK$ is a category whose objects are $C^{\ast}$-algebras and whose arrows between two objects $A$ and $B$ consist of elements of the groups $KK(A, B)$ based on \cite{H}.
The latter viewpoint aligns itself with the previous discussion of the categories in the second Section. There is also a generalized tensor product (the Kasparov product) for $KK$-theory
\[
\otimes \colon KK(A, B) \times KK(B, C) \to KK(A, C)
\]

which furnishes composition of arrows in the category $KK$. Another related theory we will also make use of is $E$-theory, also based on \cite{H}, which yields a universal extension of $KK$ that additionally fulfills excision and
is the natural receptable for pushforward, defined via strong deformations of $C^{\ast}$-algebras. There is a natural map $KK(A, B) \to E(A, B)$ which is an isomorphism if $A$ is nuclear.
We will recall the essential properties of $KK$ in Section \ref{geoBC}, in the context of the more general $\G$-equivariant $KK$-theory of Le Gall.

\subsubsection{Continuous fields of $C^{\ast}$-algebras}

For a given $C^{\ast}$-algebra $A$ we denote by $Z(A)$ the center and by $M(A)$ the multiplier algebra, i.e. the maximal unital
$C^{\ast}$-algebra which contains $A$ as an essential ideal. We denote by $T$ a locally compact Hausdorff topological space.

\begin{Def}
A $C_0(T)$-algebra is a tuple $(A, \theta)$ where $\theta \colon C_0(T) \to Z M(A)$ is a $\ast$-homomorphism such that $\theta(C_0(T)) A = A$.
\label{Def:field}
\end{Def}

An element $a \in A$ of a $C_0(T)$-algebra can be identified with a family $a = (a_x)_{x \in T}$. Here $a_x \in A_x := A / C_x A, \ C_x := \{f \in C_0(T) : f(x) = 0\}$. 
The action of functions on $T$ is implemented by $\theta$ and we often abuse notation by writing $f \cdot a$ instead of $\theta(f) \cdot a$. 
These $C_0(T)$-algebras will be used below in the definition of the \emph{analytic} $K$-homology of Lie groupoids following Le Gall's work.
On the other hand we will be interested in \emph{continuous fields} of $C^{\ast}$-algebras. 

\begin{Def}
A $C_0(T)$-algebra $A$ is \emph{continuous} if $x \mapsto \|a_x\| \in [0,\infty)$ is continuous for each $x \in T$.  
\label{Def:cont}
\end{Def}

A continuous field of $C^{\ast}$-algebras $(A_t)_{t \in [0,1]}$ is called a \emph{strong deformation} if $A_0 = A$ and $A_t = B$ for each $0 < t \leq 1$ for $C^{\ast}$-algebras $A$ and $B$.
The strong deformations can be used to define so-called asymptotic morphisms which are directly used in the construction of the $E$-theory $E(A, B)$, cf. \cite{C}, II.B.$\alpha$ for the details of this correspondence between strong deformations and
asymptotic morphisms.

We recall next the definitions needed for the analytic K-homology. We recall these notions from the work of Le Gall \cite{LeGall} where the $\G$-equivariant generalization of Kasparov's theory for any Lie groupoid $\G$ is developed in detail.

\begin{Def}
Let $\G \rightrightarrows \Gop$ be a Lie groupoid. A $\G$-algebra is a $C_0(\Gop)$-algebra endowed with a right $\G$-action
which is implemented by a family of $\ast$-isomorphisms $\alpha_{\gamma} \colon A_{r(\gamma)} \iso A_{s(\gamma)}$, parametrized by $\gamma \in \G$,
such that $\alpha_{\gamma \eta} = \alpha_{\gamma} \circ \alpha_{\eta}$ for any $(\gamma, \eta) \in \Gpull$.
\label{Def:Galg}
\end{Def}

Given a $\G$-algebra $A$ we set $(s^{\ast} A)_c = C_c(\G) s^{\ast} A$. The latter algebra is endowed with a $\ast$-algebra
structure. First fix a right Haar system $(\mu_x)_{x \in \Gop}$ and define the $\ast$-product
\[
(a \ast b)_{\gamma} = \int_{\G_{s(\gamma)}} a_{\eta} \alpha_{\eta}(b_{\gamma \eta^{-1}}) \,d\mu_{s(\gamma)}(\eta). 
\]

Also define $(a^{\ast})_{\gamma} = \alpha_{\gamma}((a_{\gamma^{-1}})^{\ast})$. We denote the resulting involutive $\ast$-algebra by $\D_{A}$. 
Consider $\|a\|_1 := \max\{|a|_1, |a^{\ast}|_1\}$ where $|a|_1$ is defined as
\[
\sup_{x \in \Gop} \int_{\G_x} \|a_{\gamma}\|_1 \,d\mu_x(\gamma). 
\]
Set $\D_{A}^1 := \overline{\D_{A}}^{\|\cdot\|_1}$ which furnishes an involutive Banach $\ast$-algebra. 
We obtain a regular representation $\lambda_x \colon \D_A^1 \to \L(L^2(G_x; A_x))$ of $\D_A^{1}$ on the Hilbert $A_x$-module $L^2(\G_x; A_x) = L^2(\G_x, \mu_x) \otimes A_x$ via
\[
\lambda_x(a)(\xi)(\gamma) = (a \ast \xi)(\gamma) = \int_{\G_{s(\gamma)}} a_{\eta} \alpha_{\eta}(\xi(\gamma \eta^{-1}))\,d\mu_{s(\gamma)}(\eta). 
\]

We set $\|f\|_r := \sup_{x \in \Gop} \|\lambda_x(f)\|$ which is called the \emph{reduced norm} as opposed to the full norm $\|f\|$ which
is the $\sup$ over all representations of $\D_{A}^1$ on $L^2(\G; A)$. 
Define $A \rtimes_r \G := \overline{\D_A^1}^{\|\cdot\|_r}$ as the reduced $C^{\ast}$-algebra associated to the $\G$-algebra $A$.
Similarly, we set $A \rtimes \G = \overline{\D_A^1}^{\|\cdot\|}$ for the full $C^{\ast}$-algebra associated to the $\G$-algebra $A$.

\begin{Rem}
Notice that our previous definition of $C_r^{\ast}(\G)$ is recovered as $C_0(\Gop) \rtimes_r \G$. 
\label{Rem:Galg}
\end{Rem}

\begin{Def}
Consider two $\G$-algebras $A$ and $B$. Let $\E$ be a Hilbert $A$-$B$ bimodule. Then a right $\G$-action on $\E$ is defined 
by a family of unitaries $V_{\gamma} \colon \E_{r(\gamma)} \to \E_{s(\gamma)}$ such that $V_{\gamma \eta} = V_{\gamma} \circ V_{\eta}$ for each $(\gamma, \eta) \in \Gpull$. 
\label{Def:Gmod}
\end{Def}

The $\G$-invariant Hilbert bi-modules are what we need to describe the $\G$-invariant cycles which are used to build up the 
$\G$-equivariant $K$-homology theory of Le Gall. While the notion of $\G$-action on $C^{\ast}$-algebras and Hilbert modules as defined above is sufficient for our purposes
we should mention that there is also a notion of \emph{generalized action} of the given groupoid via Morita equivalences. We refer to \cite{TLX} for more details concerning such generalizations.

\subsubsection{Pushforward via deformations}

We recall the definition of the index groupoid from \cite{C} which gives rise to the pushforward of $\G$-invariant smooth maps between $\G$-spaces, where $\G$ is a Lie groupoid.
This construction of the pushforward operation is useful for the definition of the geometric assembly map. 

Consider a linear map $L \colon E \to F$ between two finite dimensional vector spaces $E$ and $F$.
We denote by $F \rtimes_L E \rightrightarrows F \times \{0\}$ the action groupoid obtained from the action by translation of $E$ viewed as an \emph{additive group}
on the \emph{space} $F$, i.e. 
\[
\begin{tikzcd}[every label/.append style={swap}]
F \ar{d}{q} \ar[symbol=\circlearrowleft]{r} & (E, +) \\
\{0\} 
\end{tikzcd}
\]

given by $(x, a) \cdot (x', a') = (x, a + a')$ if $x\cdot a = x + L(a) = x'$ for $x \in F, \ a \in E$. 
By considering $F \rtimes_{t L} E$ for $t \in [0,1]$ this gives rise to a strong deformation since $C^{\ast}(F \rtimes_0 E) \cong C_0(F \times E^{\ast})$
and $F \rtimes_{tL} E \cong F \rtimes_L E$ for $t \in (0, 1]$. 
The generalization to the case of $E, F$ being smooth vector bundles and $L$ being a vector bundle map is immediate. 

If $f \colon M \to N$ is a $C^{\infty}$-map then one can define the pushforward $f_{!}$ using a linear deformation of the above type. 
In fact we will obtain a canonical element $f_{!} \in E(T^{\ast} M \oplus f^{\ast} TN, N)$. 
Let us recall how the linear deformation is applied to obtain the element $f_{!}$. Consider the vector bundle map
$df \colon TM \to f^{\ast}(TN)$. We obtain the groupoid $T_{f(x)} N \rtimes_{df_x} T_x M \rightrightarrows T_{f(x)} N \times \{0\}$
for each $x \in M$. We obtain a smooth \emph{index groupoid} 
\[
\mathrm{Ind}(df) := \bigcup_{x \in M} T_{f(x)} N \rtimes_{df_x} T_x M \rightrightarrows f^{\ast}(TN).
\]
We glue together two groupoids $\G_1 := \mathrm{Ind}(df)$ and $\G_2 := N \times (M \times M) \times (0,1]$.
The structure of the latter groupoid is obtained by combining the space $N$, viewed as a groupoid with the trivial groupoid $M \times M$ (equivalent to a point) and
the set $(0,1]$, viewed as a groupoid. The glued groupoid $\G = \G_1 \cup \G_2$ has the locally compact topology with $\G_1$ as a closed subset:
Let $(t_n, (x_n, y_n), \epsilon_n)_{n \in \Nn}$ be a sequence in $\G_2$ such that $\epsilon_n \to 0$, then 
\[
a_n \to (x, \eta, \xi) \in \G_1, \ \text{for} \ x \in M, \ \eta \in T_{f(x)} N, \ \xi \in T_x M, \ n \to \infty
\]
if and only if
\[
x_n \to x, \ y_n \to x, \ t_n \to f(x), \ \frac{x_n - y_n}{\epsilon_n} \to \xi, \ \frac{t_n - f(x_n)}{\epsilon_n} \to \eta.
\] 
The smooth structure of the groupoid $\G$ is obtained by the usual transport of structure argument, after fixing Riemannian metrics, using the exponential mappings of $M$ and $N$.
Using $\mathrm{Ind}(\epsilon df)$ for $\epsilon \in [0,1]$ as above, we obtain a canonical deformation
\[
\delta_{df} \in E(C_0(f^{\ast} TN \times T^{\ast} M), C^{\ast}(\mathrm{Ind}(df))).
\] 
Hence $C^{\ast}(\G)$ yields a canonical (strong) deformation of $C^{\ast}(\mathrm{Ind}(df))$ into $C^{\ast}(N \times (M \times M))$. 
Note that $C^{\ast}(N \times (M \times M)) \cong C(N) \otimes \K$. The latter $C^{\ast}$-algebra is Morita equivalent to $C(N)$.
Thus we have described a canonical element $f_{!} \in E(T^{\ast} M \oplus f^{\ast}(TN), N)$. 
If in addition $f$ is assumed to be \emph{$K$-oriented}, i.e. if we have a $\mathrm{spin}^c$-structure, $\sigma \in E(M, T^{\ast} M \oplus f^{\ast}(TN))$, then we can modify $f_{!}$ to obtain a canonical element $f_{!} \in E(M, N)$. 
This element $f_{!}$ only depends on the $K$-oriented homotopy class of $f$.  
We then have the wrong-way functorialy of the pushforward, i.e. the rules
\[
(\id_M)_{!} = 1 \in E(M, M), \ (f_2 \circ f_1)_{!} = f_{2!} \circ f_{1!} \in E(M_1, M_3)
\]

for $f_1 \colon M_1 \to M_2$ and $f_2 \colon M_2 \to M_3$ smooth and $K$-oriented. 

\subsubsection{Universal space for proper $\G$-actions}

Given a Lie groupoid $\G \rightrightarrows \Gop$ we define the category $\C_{\G}$ which consists of the objects \emph{proper} (right-) $\G$-manifolds
and the smooth and $\G$-invariant maps as arrows between objects. The space of $\G$-invariant smooth maps $C^{\infty}(X_1, X_2)^{\G}$ between two 
proper $\G$-spaces $X_1, X_2$ 
\[
\begin{tikzcd}[every label/.append style={swap}]
X_j \ar{d}{q_j} \ar[symbol=\circlearrowleft]{r} & \G \\
\Gop 
\end{tikzcd}
\]
consists of smooth maps $f \colon X_1 \to X_2$ such that
\[
q_2(f(x_1)) = q_1(x_1), \ x_1 \in X_1 \ \text{and} \ f(x_1 \gamma) = f(x_1) \gamma, \ (x_1, \gamma) \in X_1 \ast \G.
\]
The \emph{classifying space} $\BG$ is defined as a final object in the category $\C_{\G}$, with the \emph{universal property} 
that for any $Z \in \C_{\G}$ there is a unique (up to $\G$-invariant homotopy) $\G$-invariant map $Z \to \BG$. 

Given a $\G$-space $X \in \C_{\G}$ we define the tangent space $T_{\G}(X) := \bigcup_{x \in X} T X_x$. This yields again a $\G$-space.
We may fix a $\G$-invariant Riemannian metric on this space and identify $T_{\G}^{\ast}(X) \cong T_{\G}(X)$. 


\subsection{Geometric assembly}
\label{geoBC}

We introduce geometric cycles (after Connes, \cite{C}) for Lie groupoids. Fix a Lie groupoid $\G \rightrightarrows \Gop$. 
Let $X_1, X_2 \in \C_{\G}$ and $f \colon X_1 \to X_2$ be a $\G$-invariant smooth map. Note that $d f_x$ is $\G$-equivariantly $K$-oriented. Hence we obtain, using the index groupoid construction, an element $(d f_x)_{!} \in E(T X_{1,x}, T X_{2, x})$. In particular $(df)_{!} \in E(C^{\ast}(T_{\G} X_1 \rtimes \G), C^{\ast}(T_{\G} X_2 \rtimes \G))$ induces a map in $K$-theory
\[
K(C^{\ast}(T_{\G} X_1 \rtimes \G)) \to K(C^{\ast}(T_{\G} X_2 \rtimes \G)). 
\]
We can now define the cycles and equivalence relation for geometric $K$-homology. 
\begin{Def}
A \emph{geometric cycle} is a tuple $(X, x)$, where $X \in \C_{\G}$ and $x \in K(C^{\ast}(T_{\G}(X) \rtimes \G))$. Two geometric cycles $(X_j, x_j)$ for $j = 1,2$ are equivalent, i.e. $(X_1, x_1) \sim (X_2, x_2)$ if and only if there is a $\G$-space $X \in \C_{\G}$ and $\G$-invariant smooth maps $g_{j} \colon X_j \to X$ such that $(d g_1)_{!}(x_1) = (d g_2)_{!}(x_2)$. 
\label{Def:geoGcycles}
\end{Def}


The \emph{geometric assembly} is the homomorphism 
\[
\mu_{\G}^{geo} \colon K_{\ast}^{geo}(\G) \to K_{\ast}(C^{\ast}(\G))
\]
given by $\mu_{\G}^{geo}([X, x]) = \pi_{Z!}(x)$, where $\pi_Z \colon X \to Z$ is the $\G$-invariant map for $Z$ being a final object in the category $\C_{\G}$. 
In order words $\mu_{\G}^{geo}([X, x])$ is the image of the above deformation. The assembly map is well-defined, i.e. it depends only on the equivalence class of $(X, x)$. This follows by the functoriality of the pushforward map. 

We can also describe $\mu^{geo}$ via a strong deformation which furnishes the pushforward in $E$-theory. Let $X \in \C_{\G}$, then consider the deformation groupoid
\[
\H_{\rtimes}^t = \bigcup_{x \in \Gop} (X_x \times X_x)^{ad} \rtimes \G = (X \ast_q X) \rtimes \G \times (0,1] \cup T_{\G}(X) \rtimes \G \times \{0\} \rightrightarrows X \times [0,1].
\]

The smooth structure is obtained by glueing the two groupoids using the $\G$-invariant exponential mapping defined on $T_{\G}(X)$. 
Denote by $e_0 \colon C^{\ast}(\H_{\rtimes}^t) \to C^{\ast}(T_{\G}(X) \rtimes \G)$ the evaluation at $t = 0$ and by $e_1 \colon C^{\ast}(\H_{\rtimes}^t) \to C^{\ast}((X \ast_q X) \rtimes \G)$ the evaluation at $t \not= 0$. 
Consider these as classes $[e_0]$ and $[e_1]$ in $E$-theory. We obtain a diagram of arrows in the category $E$:
\[
\xymatrix{
C^{\ast}(\H_{\rtimes}^t) \ar@{-->}[d]_{[e_0]} \ar@{-->}[dr]^{[e_1]} & \\
C^{\ast}(T_{\G}(X) \rtimes \G) \ar@{-->}[r]_{\partial} & C^{\ast}((X \ast_q X) \rtimes \G). 
}
\]

\begin{Thm}[cf. \cite{C}, \cite{L}, \cite{CW}]
\emph{i)} If $Z \in \C_{\G}$ is a final object, then $K_{\ast}^{geo}(\G) \cong K_{\ast}(C^{\ast}(T_{\G}(Z) \rtimes \G))$. 

\emph{ii)} For each $X \in \C_{\G}$ the $C^{\ast}$-algebra $A = C^{\ast}(\H_{\rtimes}^t)$ yields a strong deformation, i.e. $A_0 = C^{\ast}(T_{\G}(X) \rtimes \G)$ and $A_t = C^{\ast}((X \ast_q X) \rtimes \G), \ t \in (0,1]$. 
Additionally, the groupoid $C^{\ast}(T_{\G}(X) \rtimes \G)$ is amenable. There is a Morita equivalence $C^{\ast}((X \ast_q X) \rtimes \G) \simM C^{\ast}(\G)$. 
\label{Thm:mugeo}
\end{Thm}

\begin{proof}
\emph{i)} Let $(X, x)$ be a geometric cycle over $\G$. Since $Z$ is a final object we fix the $\G$-equivariant map $\pi_Z \colon X \to Z$. 
We have $\mu_{Z}^{geo}[X, x] = \pi_{Z!}(x)$. Let $y \in K(C^{\ast}(T_{\G}(X) \rtimes \G))$ and define $\beta_Z(y) = [Z, y] \in K_{\ast}^{geo}(\G)$
as the class of the cocycle $(Z, y)$. Then $\mu_Z^{geo}(\beta_Z(y)) = \pi_{Z!}(y) = y$ and $\beta_Z(\mu_Z([X, x])) = [Z, \pi_{Z!}(x)] = [X, x]$.  

\emph{ii)} The strong deformation holds since $\H_{\rtimes}^t$ is a Lie groupoid and $T_{\G}(X) \rtimes \G$ is amenable by \cite{LR}.
Denote by $s, r$ the source and range map of the groupoid $T_{\G}(X) \rtimes \G$. Then $\im(s \oplus r) = \sim$ where $\sim$ is the equivalence relation
from the proper action of $\G$ on $X$, i.e. $x \sim y \Leftrightarrow y = x \gamma$ for some $\gamma \in \G$. By the properness of the action the 
groupoid $\im(s \oplus r)$ is amenable. Set $(X \rtimes \G)_x^x = G_x$ and note that $(T_{\G}(X) \rtimes \G)_x^x = T_{\G}(X)_x \rtimes G_x$. 
Again by the properness of the action of $\G$ we obtain from Proposition \ref{Prop:proper}, \emph{iii)} that $G_x$ is compact. 
Hence $T_{\G}(X)_x \rtimes G_x$ is amenable as a semi-direct product of amenable groups. Altogether it follows that $T_{\G}(X) \rtimes \G$ is amenable. 
Finally, the strict morphism $(X \ast_q X) \rtimes \G \to \G, \ (x, y, \gamma) \mapsto \gamma$ yields a Morita equivalence. 
\end{proof}

\begin{Cor}
\emph{i)} If $Z \in \C_{\G}$ is a final object the geometric assembly map can be expressed in the form $\mu_Z^{geo} = [e_0]^{-1} \otimes [e_1]$ for the classes in $E$-theory $[e_0]$ and $[e_1]$ and $\otimes$ denoting the composition.  

\emph{ii)} The geometric assembly is Morita-invariant, i.e. given a generalized isomorphism $\varphi \colon \H \dashrightarrow \G$ in the category $\LG_b$ (induced by a Morita equivalence),
then the following diagram commutes:
\[
\xymatrix{
K_{\ast}^{geo}(\H) \ar[d]_{\mu^{geo}} \ar@{>->>}[r]^{\varphi^{\ast}} & K_{\ast}^{geo}(\G) \ar[d]_{\mu^{geo}} \\
K_{\ast}(C^{\ast}(\H)) \ar@{>->>}[r]^{\varphi^{\ast}} & K_{\ast}(C^{\ast}(\G)).
}
\]
\label{Cor:mugeo}
\end{Cor}

\subsection{Analytic assembly}

The analytic assembly map $\mu_{\G}^{an}$ maps the analytic $K$-homology of the Lie groupoid $\G$ into the $K$-theory of the (reduced or full) $C^{\ast}$-algebra of $\G$. 
Following \cite{LeGall} we define the cycles for $\G$-equivariant $KK$-theory. If $A, B$ are two $\G$-algebras an $A$-$B$ bimodule $(\E, \pi)$
is a Hilbert $B$-module which is $\Zz / 2 \Zz$.graded with a $\G$-action and a $\G$-equivariant representation $\pi \colon A \to \L(\E)$. We use the notation $a \xi = \pi(a)(\xi)$.

\begin{Def}
Let $A, B$ be $\G$-algebras. A $\G$-equivariant Kasparov $A$-$B$ cycle is a triple $(\E, \pi, F)$ with a $\G$-equivariant $A$-$B$
bimodule $(\E, \pi)$ and a homogenous degree one operator $F \in \L(\E)$ such that the following conditions hold.

\hspace{0.5cm} \emph{i)} $a(F - F^{\ast}) \in \K(\E)$. 

\hspace{0.5cm} \emph{ii)} $a(F^2 - I) \in \K(\E)$. 

\hspace{0.5cm} \emph{iii)} $[a, F] \in \K(\E)$.

\hspace{0.5cm} \emph{iv)} $a_{\gamma} (F_{s(\gamma)} - V_{\gamma} F_{r(\gamma)} V_{\gamma}^{\ast}))_{\gamma \in \G} \in s^{\ast} \K(\E)$ for each $a \in s^{\ast} A$. 

We denote by $E_{\G}(A, B)$ the set of $\G$-equivariant cycles. 

\label{Def:Gcycles}
\end{Def}


Given two $\G$-algebras $A$ and $B$ we denote by $KK_{\G}(A, B)$ the group of homotopy classes of elements of $E_{\G}(A, B)$.
Note that homotopies of $E_{\G}(A, B)$ are the elements of $E_{\G}(A, B[0,1])$. We recall some properties of the bifunctor $KK_{\G} \colon C^{\ast} \times C^{\ast} \to \Ab$. 

There is a bilinear operation, the Kasparov product
\[
\otimes \colon KK_{\G}(A, C) \times KK_{\G}(C, B) \to KK_{\G}(A, B).
\]

Setting $KK_{\G}^n(A, B) := KK_{\G}(A \otimes C_0(\Rr^n), B)$ equivariant Bott periodicity holds, i.e. there is an isomorphism $KK_{\G}^{n+2}(A, B) \cong KK_{\G}^n(A, B)$. 
Additionally, $KK_{\G}$ theory is functorial with regard to generalized morphisms of Lie groupoids. 
\begin{Thm}
Let $\varphi \colon \H \dashrightarrow \G$ be a generalized morphism of Lie groupoids, then there is a natural induced homomorphism of groups $\varphi^{\ast} \colon KK_{\G}(A, B) \to KK_{\H}(\varphi^{\ast} A, \varphi^{\ast} B)$. 
\label{Thm:KKG}
\end{Thm}

This result entails many interesting consequences, e.g. equivariant versions of the Thom isomorphism, cf. \cite{CW}. 

\begin{Rem}
If we fix a (right-) Haar system $(\mu_{x})_{x \in \Gop}$ on a Lie groupoid $\G \rightrightarrows \Gop$ there are group homomorphisms, depending on the Haar system
\[
j_{\G} \colon KK_{\G}(A, B) \to KK_{\G}(A \rtimes \G, B \rtimes \G)
\]

which are natural with respect to the Kasparov product. This means for each $\G$-algebra $C$ and $x \in KK_{\G}(A, C), \ y \in KK_{\G}(C, B)$ we have $j_{\G}(x \otimes_{C} y) = j_{\G}(x) \otimes_{C \rtimes \G} j_{\G}(y)$. Additionally, $j_{\G}$ is compatible with inductive limits. 

Given a $\G$-space $X$ and a $\G$-invariant and $\G$-compact subset $Y \subset \BG$ there is an element $\lambda_{Y \times \G} \in KK(\Cc, C_0(X) \rtimes \G)$ such that we have an induced homomorphism of groups 
\[
\lambda_{Y \times \G} \otimes - \colon KK_{\ast}(C^{\ast}(Y \rtimes \G), C_0(X) \rtimes \G) \to K_{\ast}(C_0(X) \rtimes \G). 
\]

The element $\lambda_{Y \times\G}$ is also compatible with inductive limit. We refer to \cite{T2} for the construction of these homomorphisms using a cutoff function. 
\label{Rem:KKG}
\end{Rem}

\begin{Def}
The \emph{analytic $K$-homology} of a Lie groupoid $\G \rightrightarrows \Gop$ is defined by 
\[
K_{\ast}^{an}(\G) := \varinjlim_{Y \subset \BG} KK_{\G}^{\ast}(C_0(Y), C_0(\Gop))
\]

where the inductive limit runs over all $\G$-invariant and $\G$-compact subsets $Y$ of $\BG$.
\label{Def:Kan}
\end{Def}

We obtain the analytic assembly map 
\[
\mu^{an} \colon K_{\ast}^{an}(\G) \to K_{\ast}(C^{\ast}(\G)) 
\]
by an application of Remark \ref{Rem:KKG} to the $\G$-space $X = \Gop$ and taking inductive limits on both sides. 

\subsection{Comparison}

The analytic and geometric $K$-homology can be compared in a diagram involving also the geometric and analytic assembly maps. 
Following \cite{CW} we recall here the definition of the \emph{comparison map} between geometric and analytic $K$-homology of groupoids.

Let $X \in \C_{\G}$ with charge map $q$, then $q_{!} \in KK_{\G}^{\ast}(T_{\G}(X), \Gop)$. Fix also the classifying map
$c \colon T_{\G}(X) \to Y \subset \BG$. The classifying map induces an element $c_X \in KK_{\G}(Y, T_{\G}(X))$. We define $\lambda_{\G}([X, x]) = [c_X \otimes_{T_{\G}(X)} q_{!}]$
which yields a well-defined homomorphism $\lambda_{\G} \colon K_{\ast}^{geo}(\G) \to K_{\ast}^{an}(\G)$ after taking the inductive limit
over all $\G$-compact and $\G$-invariant subsets $Y \subset \BG$.

This leads to the intriguing diagram from \cite{CW}:
\begin{Thm}
The following diagram commutes
\[
\xymatrix{
K_{\ast}^{geo}(\G) \ar[dr]_{\mu_{\G}^{geo}} \ar[r(1.6)]^{\lambda_{\G}} & & K_{\ast}^{an}(\G) \ar[dl]_{\mu_{\G}^{an}} \\
& K^{\ast}(\G) & 
}
\]
\label{Thm:comparison}
\end{Thm}

\begin{Rem}
It is an interesting question for which Lie groupoids $\lambda_{\G}$ yields an isomorphism.
Additionally, to generalize other known constructions from the Baum-Douglas theory, see also \cite{MN}. Including applications of the comparison between the geometric and analytic model for $K$-homology to the index
theory of $\G$-equivariant operators for any Lie groupoid $\G$.  
\label{Rem:comparison}
\end{Rem}
%

\begin{table}[ht]
\caption{Baum-Connes conjecture Hausdorff groupoids}
\centering
\begin{tabular}{c c c}
\hline \hline
Lie groupoid & $\mu^{an}$ & References \\ [0,5ex]
\hline 
Bolic groups & injective & Tu, \cite{T1} \\
Haagerup (hence amenable) & isomorphism & Tu, \cite{T2} \\
almost connected Lie groups & isomorphism & Chapert-Echterhoff-Nest, \cite{CEN} \\
Hyperbolic groups & isomorphism & Lafforgue, \cite{Laf} \\ [1ex]
\hline
\end{tabular}
\label{table:BC}
\end{table}

\textbf{Counterexamples:} There are counterexamples for the Baum-Connes conjecture \emph{with coefficients} as well as for the case of non-amenable groupoids, cf. \cite{HLS}. 

\textbf{Problem:} Define and study the analytic assembly mapping $\mu^{an}$ for non-Hausdorff groupoids.

\section{Index theory of singular manifolds}
\label{Singmf}





In this section we study an extension of Atiyah-Singer index theory to singular manifolds. We consider a Lie manifold with
additional (regular) embedded, transverse boundary stratum. This setup makes it possible to address problems related to the qualitative
studies of partial differential equations on a singular manifold, with additional boundary conditions. We will see that the basic
techniques used are in the same spirit as in the proof of the index theorem for the base case considered in Section \ref{proof}, i.e. the case where the Lie manifold $M$
is a compact manifold without boundary and the Lie structure consists of all vector fields $\V = \Gamma(TM)$.

Let us first define what is called a \emph{Lie manifold with boundary} in \cite{AIN}. 
As before, a \emph{Lie manifold with boundary} consists of a manifold $M$ with corners, whose boundary
hyperfaces are $F_1, \dots, F_k$, i.e.\ $M=M_0 \cup F_1 \cup\cdots\cup F_k$. However, one boundary hyperface,
called the \emph{regular boundary}, say  $Y:= F_1$, has a special role, similar to a boundary component of a Riemannian manifold with boundary.
If we glue two copies of $M$ along $Y$, then we obtain the double $M \cup_{Y} M$. We do this doubling
such that the boundary of the double at the interior of $Y \cap F_i$, $i>1$ has no corner.
Roughly speaking, a Lie structure with boundary on $M$ is defined as the restriction of a Lie structure on $M \cup_{Y} M$
to one of the copies of $M$, see \cite{AIN} for details. If $Y_0$ is the interior of $Y$, then the map $\varrho$ defines a bundle isomorphism from $\A|_{Y_0}$ to $TM|_{Y_0}$, and
$Y$ carries an induced Lie structure, denoted by $\B$. In the following $\W:=\Gamma(\B)\subset \Gamma(TY)$.

The formal definition is given as follows.
\begin{Def}
A \emph{Lie manifold with boundary} is a Lie manifold $(M, \A, \V)$ together with a submanifold $(Y, \B, \W)$ such that the following
conditions hold:

\emph{1)} $(Y, \B, \W) \hookrightarrow (M, \A, \V)$ is a \emph{Lie submanifold}, i.e. $Y \subset M$ is a submanifold with corners where
$\B \to Y$ is a $C^{\infty}$-vector bundle such that $\Gamma(\B) \cong \W$ and $\B \hookrightarrow \A_{|Y}$ is a Lie subalgebroid.

\emph{2)} The submanifold $Y$ is \emph{transverse} in $M$, i.e. $T_p M = \mathrm{span} \{\varrho(\A_p), T_p Y\}$ for each $p \in \partial Y = Y \cap \partial M$. 

\label{Def:Liebdy}
\end{Def}

\begin{Rem}
\emph{i)} Condition \emph{2)} is equivalent to $T_x M = T_x Y + T_x F$ for each $x \in F \cap Y$ and each closed codimension one face $F \in \F_1(M)$. 
Notice that the Lie structure of vector fields $\W$ is - by definition of a Lie submanifold - a subalgebra of $\V_{|Y}$, precisely $\W = \{V_{|Y} : V \in \V, \ V_{|Y} \ \text{tangent to} \ Y\}$. 

\emph{ii)} Given a Lie manifold with boundary as above. Consider the exponential map $\exp \colon \A \to M$ which is the natural extension from the interior $\exp \colon TM_0 \to M_0$. 
Setting $\N := \frac{\A_{|Y}}{\B}$ the $\A$-normal bundle, where $\Gamma(\B) \cong \W$ and $\B \hookrightarrow \A_{|Y}$ is in particular a sub vector bundle.
We obtain the exact sequence of vector bundles
\[
\xymatrix{
\B \ar@{>->}[r] & \A_{|Y} \ar@{->>}[r]^{q} & \N.
}
\]

This sequence splits and denote by $\eta \colon \N \to \A_{|Y}$ the splitting. By a choice of an $\A$-metric, we can
choose $\eta$ as an isomorphism $\eta \colon \N \iso \B^{\perp}, q \circ \eta = \id_{\N}$. This furnishes the decomposition 
$\A_{|Y} \cong \B \oplus \N$. 

\emph{iii)} An application of the achnor $\varrho$ of the Lie algebroid $\A$ yields an isomorphism
\[
\frac{\A_p}{\B_p} = \N_p \iso \frac{T_p M}{T_p Y} \cong N_p Y, \ p \in Y. 
\]

Hence, in particular, $\N_{|Y_0} \cong N Y_0$ is an isomorphism over $Y_0$. 
\label{Rem:Liebdy}
\end{Rem}

We first need to think about the correct definition of the analytic index and of the topological index for Lie manifolds with boundary.
The first problem is easier to solve than the second. We already know what the generalized analytic index looks like for a Lie manifold without (regular) boundary using a natural generalization of Connes tangent groupoid. 
For the extension to the case with boundary we will rely on a semi-groupoid construction, based on \cite{ANS}. The second problem, of defining an appropriate version of the topological index, is more intricate
and is not completely resolved for general types of singular foliations. We will base our definition of a topological index on a classifying space construction for manifolds with corners, cf. \cite{MN2}. 


\subsection{The case without boundary}

Following \cite{MN2} we recall the definition of the \emph{classifying space} for manifolds with corners and sketch the construction of the topological
index for manifolds with corners. We use this in the next Section to extend the index theorem to Lie manifolds with boundary.

Fix a Lie manifold $(M, \A, \V)$. We also assume that here that $\V$ is the Lie structure $\V_b$ of all vector fields tangent to the boundary. 
This case corresponds to a manifold with polycylindrical ends. 
Recall the definition of the adiabatic groupoid $\Gad = \A(\G) \times \{0\} \cup \G \times (0,1]$ from Remark \ref{Rem:blowup}. 
Consider the evaluation morphisms $e_0 \colon C_c^{\infty}(\G^{ad}) \to C_c^{\infty}(\A(\G))$ and $e_t \colon C_c^{\infty}(\G^{ad}) \to C_c^{\infty}(\G \times (0,1]), \ t \not= 0$. 
As in the case of the tangent groupoid (cf. Proposition \ref{Prop:SES}) we obtain a short exact sequence (cf. \cite{MP})
\[
\xymatrix{
C^{\ast}(\G) \times C_0(0,1] \ar@{>->}[r] & C^{\ast}(\G^{ad}) \ar@{->>}[r]^{e_0} & C^{\ast}(\A(\G)). 
}
\]
After application of the fiberwise Fourier transform we have the isomorphism $C^{\ast}(\A(\G)) \cong C_0(\A^{\ast}(\G))$. 
We apply the six-term exact sequence in $K$-theory and obtain an isomorphism (note that again $C^{\ast}(\G) \otimes C(0,1]$ is a contractible $C^{\ast}$-algebra) $(e_0)_{\ast} \colon K_{\ast}(C^{\ast}(\G^{ad})) \to K^{\ast}(\A^{\ast}(\G))$. 
Thus we obtain the \emph{generalized analytic index} 
\begin{align}
\ind_a &:= (e_1)_{\ast} \circ (e_0)_{\ast}^{-1} \colon K^{\ast}(\A^{\ast}(\G)) \to K_{\ast}(C^{\ast}(\G)). \label{inda}
\end{align}

\begin{Rem}
Unlike for the case $\V = \Gamma(TM)$ and $M$ a compact smooth manifold with empty boundary, the generalized analytic index $\ind_a$ does not equal the Fredholm index for the case of manifolds with corners.
Note that in the first case the construction of $\G(M)$ specializes to the pair groupoid $\G = M \times M$ and $C^{\ast}(M \times M) \cong \K(L^2(M))$. 
In general we will consider index problems with values in the $K$-theory of $C^{\ast}(\G)$. 
\label{Rem:Fhindex}
\end{Rem}

We recall the definition of a classifying space for manifolds with corners
whose existence we will not prove in this work. We remark that the explicit construction of such a space is technical.

\begin{Def}
A manifold with corners $X_M$ is called \emph{classifying space for manifolds with corners} if $M \hookrightarrow X_M$ is a closed embedding of manifolds with corners and the following conditions hold

\emph{i)} all the open faces of $X_M$ are diffeomorphic to euclidean spaces, 

\emph{ii)} for every open face $F$ of $X_M$ we obtain an open face $F \cap M$ of $M$ and each open face of $M$ is obtained in this way. 
\label{Def:classifying}
\end{Def}

The previous definition allows us to define the topological index of $M$ as defined in \cite{MN2}. This construction relies on the special nature
of the space $X_M$ and the corresponding generalized analytic index of $X_M$. 
We fix the notation $\A_{X_M} \to X_M$ to denote the $b$-tangent bundle of $X_M$, i.e. the vector bundles whose sections consist
of smooth vector fields tangent to the boundary strata of $X_M$. Consider the integrating $s$-connected Lie groupoid
$G_{X_M} \rightrightarrows X_M$ such that $\A(\G_{X_M}) \cong A_{X_M}$. Such a groupoid is the minimal (holonomy) groupoid
of the $b$-Lie structure and as such unique up to Morita equivalence. Denote by $\ind_a^{X_M} \colon K^{\ast}(\A_{X_M}) \to K_{\ast}(C^{\ast}(\G_{X_M})$ the corresponding generalized analytic index of $X_M$.

\begin{Thm}[Monthubert-Nistor]
There is a natural induced map in $K$-theory $i_K \colon K_{\ast}(C^{\ast}(\G)) \to K_{\ast}(C^{\ast}(\G_{X_M}))$ such that 
the diagram 
\[
\xymatrix{
K_{\ast}(C^{\ast}(\G)) \ar[r]^{i_K} & K_{\ast}(C^{\ast}(\G_{X_M})) \\
K^{\ast}(\A^{\ast})) \ar[u]_{\ind_a} \ar[r]_{i_{!}} & K^{\ast}(\A_{X_M}^{\ast}) \ar[u]^{\ind_a^{X_M}} 
}
\]

commutes. Additionally, $\ind_a^{X_M}$ is an isomorphism.
\label{Thm:MN}
\end{Thm}

We define the \emph{topological index} 
\begin{align}
\ind_t &:= i_K^{-1} \circ \ind_a^X \circ i_{!} \colon K^{\ast}(\A_M) \to K_{\ast}(C^{\ast} \G_M). \label{indt}
\end{align}

\begin{Rem}
The equality of the topological and generalized analytic index, as proven in \cite{MN2}, yields in particular another proof of the Atiyah-Singer index theorem
for a compact closed manifold. One can view the isomorphism induced by $\ind_a^{X_M}$ as a Baum-Connes type result. It is an
interesting problem to find a general definition of a Atiyah-Singer type topological index for any Lie groupoid. This problem seems to be strongly related to
the Baum-Connes property of the given groupoid. 
\label{Rem:MN}
\end{Rem}

\subsection{The case with boundary}

Fix a Lie manifold $(M, \A, \V)$ with boundary $(Y, \B, \W)$. We define next a Lie manifold of cylinder type.
\begin{Def}
The Lie manifold with boundary $(M, \A, \V)$ is of \emph{cylinder type} if it is diffeomorphic to cylindrical neighborhood $Y \times \Rr$ of $Y$.
\label{Def:cylinder}
\end{Def}

Let $\G \rightrightarrows M$ denote an integrating $s$-connected Lie groupoid, i.e. $\A(\G) \cong \A$. Fix the generalized exponential map (see also \cite{LR}, \cite{BS}) $\Exp \colon \A(\G) \to \G$
as well as the exponential map $\exp \colon \A(\G) \to M$ induced by an invariant connection on $\A$. 
Consider the half space $\Abdy \subset \A$ which is defined as follows:
\[
\Abdy := \{v \in \A : \exp(-tv) \in M, \ t > 0 \ \text{small}\}. 
\]
This is the natural generalization of the half-space introduced in \cite{ANS} for the case of a compact manifold with boundary and trivial Lie structures.
We restrict the invariant connection $\nabla$ of $\A$ to $\Abdy$ and also obtain the restriction of the generalized exponential which we denote by the same symbol $\Exp \colon \Abdy \to \G$.
Then we define the deformation semi-groupoid $\Gbdy \rightrightarrows M \times I$ where $I = \Rr$ or $I = [0,1]$ and $I^{\ast} := I \setminus \{0\}$ as follows
\[
\Gbdy := \G \times (0,1] \cup \Abdy \times \{0\}.
\]
Note that a priori $\Gbdy$ has a natural semi-groupoid structure: the groupoid structure of $\G$, of $(0,1]$ viewed simply as a set
as well as the semi-groupoid structure of $\Abdy$, viewed as a bundle of half-spaces.
Note that $\Gbdy \subset \Gad$ where $\Gad$ is the adiabatic groupoid $\G \times (0,1] \cup \A \times \{0\}$.
By the local diffeomorphism property of $\Exp$ we can describe a smooth structure on $\Gad$. It is defined by glueing a neighborhood $\O$ of $\A \times \{0\}$ to $\G\times I^\ast$ via
\[
\O \ni (v, t) \mapsto \begin{cases} v, \ t = 0 \\
(\Exp(-tv), t), \ t > 0. \end{cases}
\]

Then the smooth structure of $\Gbdy \subset \Gad$ is the one induced by $\Gad$ with regard to the locally compact subspace topology, i.e.
$C_c^{\infty}(\Gbdy) := C_c^{\infty}(\Gad)_{|\Gbdy}$.
The next goal is to define a continuous field of $C^{\ast}$-algebras over the Lie semi-groupoid $\Gbdy$. 

\begin{Def}
The $C^{\ast}$-algebra associated to $\Gbdy$ is defined as the completion $C_r^{\ast}(\Gbdy) := \overline{C_c^{\infty}(\Gbdy)}^{\|\cdot\|}$.
We define the norm $\|\cdot\|$ as the reduced norm with regard to the representation $\tilde{\pi} := (\pi , \pi^{\partial})$ on the Hilbert space 
$\H := L^2(\G) \oplus L^2(\Abdy_{|Y})$. Here $\pi = (\pi_t)_{0 < t \leq 1}$ where for $0 < t \leq 1$ 
\[
\pi_t(f) \xi(\gamma) = \frac{1}{t^n} \int_{\G_{s(\gamma)}} f(\eta, t) \xi(\eta) \,d\mu_{s(\gamma)}(\eta).
\]

Define the representation $\pi_{0}^{\partial}$ on the Hilbert space $L^2(\Abdy_{|Y})$ by
\[
\pi_0^{\partial}(f) \xi(v) = \int_{\Abdy_{\tilde{\pi}(v)}} f(v - w) \xi(w)\,dw.
\]

\label{Def:CGbdy}
\end{Def}

We also introduce a $C^{\ast}$-algebra associated to the half-space $\Abdy$.
\begin{Def}
Define the reduced $C^{\ast}$-algebra of $\Abdy$ in terms of the completion $C_r^{\ast}(\Abdy) := \overline{C_c^{\infty}(\Abdy)}^{\|\cdot\|_{\tilde{\pi}_0}}$,
where $\tilde{\pi}_0 = (\pi_0, \pi_0^{\partial})$ is the representation of $C_c^{\infty}(\Abdy)$ on the Hilbert space $\H := L^2(\A) \oplus L^2(\Abdy_{|Y})$.
We define
\[
\pi_0(f) \xi(v) = \int_{\A_{\pi(v)}} f(v - w) \xi(w) \,dw
\]

and
\[
\pi_0^{\partial}(f) \xi(v) = \int_{\Abdy_{\tilde{\pi}(v)}} f(v - w) \xi(w) \,dw.
\]

\label{Def:CAbdy}
\end{Def}

Note that the above definition furnishes a field of $C^{\ast}$-algebras with $\varphi_t \colon C_r^{\infty}(\Gbdy) \to C_r^{\ast}(\Gbdy)(t)$ where
$C_r^{\ast}(\Gbdy_t) = C_r^{\ast}(\G), \ t \not= 0$ and $C_r^{\ast}(\Gbdy_0) = C_r^{\ast}(\Abdy)$. We refer to \cite{B} for the proof of the following result.
\begin{Thm}
Let $(M, \A, \V)$ be a Lie manifold of cylinder type with boundary $(Y, \B, \W)$ and integrating Lie groupoid $\G \rightrightarrows M$. 
The field $(C_r^{\ast}(\Gbdy), \{C_r^{\ast}(\Gbdy_{t}), \varphi_t\}_{t \in [0,1]})$ is a continuous field of $C^{\ast}$-algebras.
\label{Thm:CGbdy}
\end{Thm}

We assume in the following that $\V = \V_b$, i.e. the maximal Lie structure of $b$-vector fields. 
We will next address the index theory of such Lie manifolds with boundary. Similarly, as the index theorem for the case without boundary is a generalization of the Atiyah-Singer index theorem, the case with
boundary should be viewed as a particular generalization of the Boutet de Monvel index formula, cf. \cite{BM}. 
First let us give the definition of the generalized analytic index for boundary value problems (BVP's). 

The short exact sequence
\[
\xymatrix{
C^{\ast}(\G) \otimes C_0(0,1] \ar@{>->}[r] & C^{\ast}(\Gbdy) \ar@{->>}[r] & C_r^{\ast}(\Abdy)
}
\]

induces the isomorphism in $K$-theory $(e_0)_{\ast} \colon K_0(C_r^{\ast}(\Gbdy)) \iso K_0(C_r^{\ast}(\Abdy))$. 
We obtain the generalized analytic index for BVP's 
\begin{align}
\widetilde{\ind_a} := (e_1)_{\ast} \circ (e_0)_{\ast}^{-1} \colon K_0(C_r^{\ast}(\Abdy)) \to K_0(C_r^{\ast}(\Gbdy)). \label{indaBVP}
\end{align}

Set $\mathring{\G} := \G_{M \setminus Y}^{M \setminus Y}$, then $K_{\ast}(C^{\ast}(\G)) \cong K_{\ast}(C^{\ast}(\mathring{\G}))$.
Hence $\ind_a$ and $\widetilde{\ind_a}$ take values in the same group. 

We prove the following equality of indices.
\begin{Thm}
The generalized analytic index for boundary value problems, $\widetilde{\ind_a}$ equals the topological index $\ind_t$, up to a $KK$-equivalence. 
\label{Thm:ASBVP}
\end{Thm}

\begin{proof}
Consider the short exact sequence 
\[
\xymatrix{
C_r^{\ast}(\A) \ar@{>->}[r] & C_r^{\ast}(\Abdy) \ar@{->>}[r] & Q. 
}
\]

The quotient $Q$ is isomorphic to $C_0(\B^{\ast}) \otimes \Tau_0$, where $\Tau_0$ is the algebra of completed Wiener-Hopf operators.
The latter algebra is obtained after choosing a trivialization in the decomposition $\A_{|Y} \cong \B \oplus \N$, which is already the case for a cylinder type Lie manifold. 
It is easy to check that $\Tau_0$ is $K$-trivial, i.e. $K_{\ast}(\Tau_0) \cong 0$. Denote by $\Phi$ the isomorphism in $K$-theory $K_{\ast}(C_r^{\ast}(\Abdy)) \cong K_{\ast}(C_0(\A^{\ast}))$. 
This yields a commuting diagram of the form 
\[
\xymatrix{
K_0(C_r^{\ast}(\Abdy)) \ar[d]_{\Phi} \ar[r]^-{\widetilde{\ind_a}} & K_0(C^{\ast}(\G)) \cong K_0(C^{\ast}(\mathring{\G})) \\
K_0(C_0(\A^{\ast})) \ar[ur]_-{\ind_a}
}
\]

Setting $\widetilde{\ind_t} := \ind_t \circ \Phi$ the above diagram yields that $\widetilde{\ind_a} = \widetilde{\ind_t}$. 
\end{proof}

\section{Relative quantization}
\label{relQuant}

In the previous section we have considered a topological index theorem for singular manifolds with boundary. We did not rely on the pseudodifferential calculus. 
We have defined the quantization for pseudodifferential operators on any Lie groupoid in Section \ref{PsDos}. In this Section we introduce an
extension of the pseudodifferential calculus to take into account boundary conditions and boundary value problems of pseudodifferential type on singular manifolds based on \cite{BS}.

In the standard case of a compact manifold with boundary (no singularities) such a calculus has been introduced by Boutet de Monvel, \cite{BM}.
Since then this calculus has been extensively studied by other authors. We base our generalization on a calculus similar in spirit to Boutet de Monvel's ideas following \cite{NS}. 

We first reformulate slightly the definition of a Lie manifold with boundary. Precisely, we consider \emph{relative} (elliptic) problems of the following type:
Given an embedding of Lie manifolds $j \colon Y \hookrightarrow M$. Endow $M$ and $Y$ with compatible Riemannian metrics $g$ and $h$ respectively. 
Consider the Laplacians $\Delta = \Delta_g$ corresponding to $g$ on $M$ and $\Delta_{\partial} = \Delta_{h}$ corresponding to $h$ on $Y$. 
Here the operators act on suitable $L^2$-based Sobolev spaces $H_{\V}^s(M)$ and $H_{\W}^s(Y)$ which are completions of $C_c^{\infty}$, cf. \cite{AIN}.
These pairs form the example of a \emph{relative elliptic operator} in terms the diagonal matrix
\[
\begin{pmatrix} \Delta_g & 0 \\ 0 & \Delta_{\partial} \end{pmatrix} \colon \begin{matrix} H_{\V}^{s_1}(M) \\ \oplus \\ H_{\W}^{s_2}(Y) \end{matrix} \to \begin{matrix} H_{\V}^{t_1}(M) \\ \oplus \\ H_{\W}^{t_2}(Y) \end{matrix}. 
\]
More generally, we would like to describe an operator algebra from the given data $j \colon Y \hookrightarrow M$ that captures
all non-commutative phenomena which arise in the study of elliptic boundary value problems (BVP's). 
Inspired by the classical BVP's (e.g. the Dirichlet and Neumann problem) we introduce additional operators.
Let $j^{\ast}$ be the restriction to $Y$ operator. 
By \cite{AIN}, Theorem 4.7 we have $j^{\ast} \colon H_{\V}^{s}(M) \to H_{\W}^{s-\frac{\nu}{2}}(Y)$ continuously for 
$s > \frac{\nu}{2}$, where $\nu$ denotes the codimension of $j(Y)$ in $M$. 
Denote by $j_{\ast}$ the formal adjoint of $j^{\ast}$. 
Unfortunately, as can be checked easily already in the standard case, operators of the form
\begin{equation}
\label{eq:opmat}
\begin{pmatrix} \Delta_g & j_{\ast} \\ j^{\ast} & \Delta_{\partial} \end{pmatrix}
\end{equation}
are not closed under composition. 
The reason is that operators in the upper left corner, the singular Green operators, of the form $P j_{\ast} j^{\ast} Q$, for
$P$ and $Q$ (pseudo-)differential operators, are not again (pseudo-)differential operators.

We describe these operators abstractly using the theory of Fourier integral operators on Lie groupoids, cf. \cite{LV}. 
For a given compact manifold with corners $M$ we fix the notation $\F(M)$ to denote the collection of boundary faces of $M$.

\begin{Def}
Let $M_i, \ i =1,2$ be compact manifolds with corners and $j \colon M_1 \hookrightarrow M_2$ a $C^{\infty}$-embedding. 

\emph{i)} The embedding is called \emph{closed} if for each boundary face $F$ of $M_2$ there is a boundary face $G$ of $M_1$ such that $F$ is the connected component of
$j^{-1}(G)$.

\emph{ii)} The embedding $j \colon M_1 \hookrightarrow M_2$ is \emph{transverse relative to $M_2$} if for any boundary face $F$ of $M_2$
and any $y \in F \cap j(M_1)$ we have that the space $T_y M$ is the non-direct sum of $T_y j(M_1)$ and $T_y F$. 

\emph{iii)} Denote by $\varphi(j) \colon \F(M_2) \to \F(M_1)$ the map of boundary faces given by $\F(M_2) \ni F \mapsto M_1 \cap F \in \F(M_1)$. 
The embedding $j \colon M_1 \hookrightarrow M_2$ is called \emph{admissible} if $\varphi(j)$ is bijective. 
\label{Def:transv}
\end{Def}

\begin{Exa}
A typical situation in which such an embedding arises is if one takes a submanifold of $\Rr^d$ that ``extends up to infinity''. If one then compactifies $\Rr^d$, the submanifold will hit the new-formed boundary at infinity. Figure \ref{fig:embex} is an example of a transverse relative embedding of compact manifolds with non-trivial corners. Here, $M$ is the ``torus with corners'' $\mathbb{S}^1\times Y$

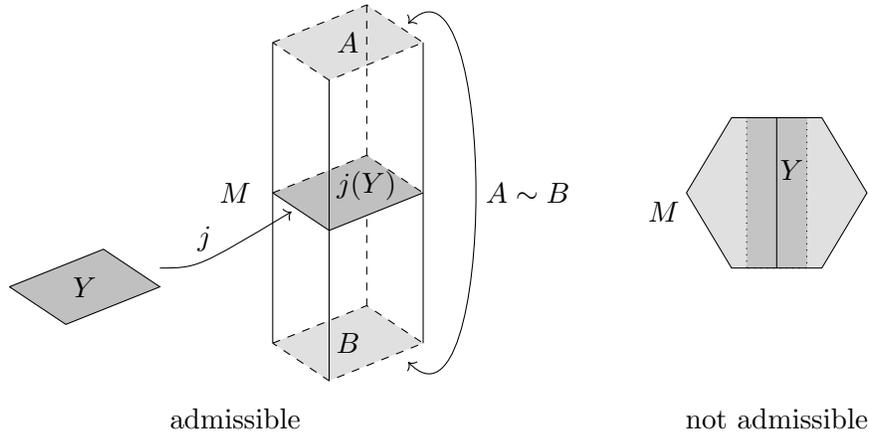
\begin{figure}[H]
\begin{center}
\begin{tikzpicture}
    \begin{scope}[shift={(-3.5,-1.25)}]	
    \fill[opacity=0.25] (0,0) -- (1.25,0.5)  -- (2,0) -- (0.75,-0.5) -- (0,0);
	\draw (0,0) -- (1.25,0.5)  -- (2,0) -- (0.75,-0.5) -- (0,0);
	\node  at (1,0) {$Y$};
	\end{scope}
	\node  at (-0.5,-3) {admissible};
	\draw[->] (-1.5,-1) .. controls (-1,-1) and (-1,-1) .. (0.25,-0.25) node [above, midway] {$j$}; 
    \fill[opacity=0.25] (0,0) -- (1.25,0.5)  -- (2,0) -- (0.75,-0.5) -- (0,0);
	\draw[dashed] (0,0) -- (1.25,0.5)  -- (2,0);
    \draw (2,0) -- (0.75,-0.5) -- (0,0);
    \node  at (-0.5,0) {$M$};
    \node  at (1.25,0.1) {$j(Y)$};
    \fill[opacity=0.125] (0,-2) -- (1.25,-1.5)  -- (2,-2) -- (0.75,-2.5) -- (0,-2);
    \fill[opacity=0.125] (0,2) -- (1.25,2.5)  -- (2,2) -- (0.75,1.5) -- (0,2);
    \draw[dashed] (0,-2) -- (1.25,-1.5)  -- (2,-2) -- (0.75,-2.5) -- (0,-2);
    \draw[dashed] (0,2) -- (1.25,2.5)  -- (2,2) -- (0.75,1.5) -- (0,2);
    \node  at (1,2) {$A$};
    \node  at (1,-2) {$B$};
    \draw[<->] (1.8,2.25) .. controls (3,3.5) and (3,-3.5) .. (1.8,-2.25) node [right, midway] {$A\sim B$}; 
    \draw (0,-2) -- (0,2);
    \draw (2,-2) -- (2,2);
    \draw (0.75,-2.5) -- (0.75,1.5);
    \draw[dashed] (1.25,2.5) -- (1.25,0.5);
    \draw[dashed] (1.25,-0.35) -- (1.25,-1.5);
    
    \begin{scope}[shift={(5.5,0)}]
        \node  at (1.4,0.3) {$Y$};
        \node  at (-0.3,-0.25) {$M$};
\draw (0,0) -- (0.6,1) -- (1.8,1) -- (2.4,0) -- (1.8,-1) -- (0.6,-1) -- (0,0);
\fill[opacity=0.125] (0,0) -- (0.6,1) -- (1.8,1) -- (2.4,0) -- (1.8,-1) -- (0.6,-1) -- (0,0);
\draw (1.2,-1) -- (1.2,1);
\draw[dotted] (0.8,-1) -- (0.8,1) -- (1.6,1) -- (1.6,-1) -- (0.8,-1);
\fill[opacity=0.125] (0.8,-1) -- (0.8,1) -- (1.6,1) -- (1.6,-1) -- (0.8,-1);    
\node  at (1.2,-3) {not admissible};
    \end{scope}
    
\end{tikzpicture}
\end{center}
\caption{Examples of transverse relative embeddings}
\label{fig:embex}
\end{figure}
\label{Exa:fig}
\end{Exa}

We examine the class of embeddings compatible with the geometric structure of Lie manifolds.

\begin{Def}
An embedding of Lie manifolds $j \colon (Y, \B, \W) \hookrightarrow (M, \V, \A)$ is a transverse embedding of manifolds with corners such that inclusion $\B \hookrightarrow \A_{|Y}$ is a Lie subalgebroid. 
We denote by $\EmbV$ the category of Lie manifolds with admissible embeddings as arrows. 
\label{Def:Liemf2}
\end{Def}

\begin{Prop}
Let $j \colon Y \hookrightarrow M$ be an embedding of Lie manifolds. 
There is a tubular neighborhood $Y \subset \U \subset M$ such that the embedding $\tilde{j} \colon Y \hookrightarrow \U$ is admissible. 
\label{Prop:adm}
\end{Prop}

\begin{proof}
This follows from the tubular neighborhood theorem given in \cite[Theorem 2.7]{AIN}.  
\end{proof}

\begin{Rem}
The assumption of admissibility is the relative version of our previous definition \ref{Def:cylinder} of a Lie manifold with boundary of cylinder type. 
This seems very restrictive at first. We mention however that in order to obtain a calculus for a relative transversal \emph{non-admissible} embedding it suffices to reduce ones attention to a tubular neighborhood $\tilde{M}$ of $j(Y)$ as illustrated in Figure \ref{fig:embex} where the embedding is admissible and then glue the resulting calculus with the corresponding pseudodifferential calculus on $M \setminus \tilde{M}$. This construction is carried out in \cite{B}.
%
\label{Rem:adm}
\end{Rem}

\subsection{Quantization}

Let us fix an admissible embedding of Lie manifolds $j \colon (Y, \B, \W) \hookrightarrow (M, \A, \V)$ as defined in the previous section.
We describe a microlocalization of such an embedding using a quantization procedure for special symbol classes suitable
for an embedding of manifolds. The quantization is a generalization of the pseudodifferential calculus on Lie groupoids previously discussed in Section \ref{PsDos}.
A thorough study of the operators obtained with the help of this quantization for a given embedding requires the use of Fourier integral operators on groupoids as introduced in \cite{LV} and
extended to groupoid actions in \cite{BS}. We merely state here the definition of the Kohn-Nirenberg quantization without going into too much detail regarding the microlocal structure of the 
resulting operator calculus.

Let $\G \rightrightarrows M$ be an integrating $s$-connected Lie groupoid, i.e. $\A(\G) \cong \A$. One can show that by admissibility
of the given embedding $j$ we obtain $\H := \G_Y^Y = r^{-1}(Y) \cap s^{-1}(Y)$ is an integrating groupoid for the Lie structure $\W$ on $Y$, i.e. $\B \cong \A(\H)$.
Additionally, admissiblity can be rephrased as a one-to-one correspondence between the orbits of $\H$ and $\G$.
Hence these groupoids are Morita equivalent. Set $Z := \G_M^Y = r^{-1}(Y)$ and $Z^t := \G_Y^M = s^{-1}(Y)$. Then there is a canonical diffeomorphism $\flip \colon Z \iso \Zop$ of manifolds with corners. Additionally, $Z$ is a right $\G$-space and
a left $\H$-space, by right and left composition respectively. The space $Z$ implements the Morita equivalence between $\H$ and $\G$.
We refer to \cite{BS} for the proofs of these assertions. Let us fix the following diagram of generalized morphisms
\[
\begin{tikzcd}[every label/.append style={swap}]
\H \ar[dotted]{rr} \ar[symbol=\circlearrowleft]{rd} & & \G \ar[dotted]{rr} \ar[symbol=\circlearrowright]{ld} \ar[symbol=\circlearrowleft]{rd} & & \H \ar[symbol=\circlearrowright]{ld} \\[-2\jot]
& \ar{ld}{p} Z \ar{rr}{\flip} \ar{rd}{q} & & \ar{ld}{p^t} \Zop  \ar{rd}{q^t} & \\
Y & & M & & Y 
\end{tikzcd}
\]

The left hand side yields a morphism in $\LG_b$ and the right hand side a morphism in the opposite category $\LG_b^{op}$.
This setup is now used to define the operators for the relative data we have given in the form of the embedding $j$.


We consider the following embeddings which are inclusions of submanifolds (with corners) $\rho \colon \H \hookrightarrow Z$ and $\sigma \colon Z \hookrightarrow \G$. 
Then we fix the normal bundles associated to these inclusions.
\begin{itemize} 
\item The normal bundle $\N \to Y$ to the inclusion $Y \hookrightarrow M$. 

\item The normal bundle $\N^{Z} Y$ with regard to the inclusion $\rho \circ u_{\partial} \colon Y \hookrightarrow Z$ where
$u_{\partial} \colon Y \hookrightarrow \H$ denotes the unit inclusion in the groupoid $\H$.

\item The normal bundle $\N^{\G} Y$ with regard to the inclusion $\sigma \circ \rho \circ u_{\partial} \colon Y \hookrightarrow \G$. 

\end{itemize}

We fix the following conormal bundles, which are also called \emph{canonical $\G$-relations} in the theory of Fourier integral operators on groupoids and correspondences. 
\begin{align*}
\Lambda_{\Psi} &:= \A^{\ast}(\G) \subset T^{\ast} \G \setminus 0, \\
\Lambda_{\partial} &:= \A^{\ast}(\H) \subset T^{\ast} \H \setminus 0, \\
\Lambda_g &:= (\N^{\G} \Delta_Y)^{\ast} \subset T^{\ast} \G \setminus 0, \\
\Lambda_b &:= (\N^{Z} \Delta_Y)^{\ast} \subset T^{\ast} Z \setminus 0, \\
\Lambda_c &:= (\N^{Z^t} \Delta_Y)^{\ast} \subset T^{\ast} Z^t \setminus 0. 
\end{align*}

We have the non-canonical decompositions as vector bundles
\begin{align}
& \N^{Z} \Delta_Y \cong \A_{\partial} \oplus \N, \label{decomp1} \\
& \N^{Z^t} \Delta_Y \cong \A_{|Y} \oplus \N \cong \A_{\partial} \oplus \N \oplus \N, \label{decomp2}. 
\end{align}

We call $\W = \Gamma(\A(\H))$ the \emph{tangent} vector fields, $\Gamma(\N_1)$ transversal \emph{of the first kind} and $\Gamma(\N_2)$ transversal \emph{of the second kind}. 
Fix the projections $\pi_b \colon (\N^{Z} Y)^{\ast} = \B^{\ast} \times \N^{\ast} \to \B^{\ast}$, $\pi_{g}^{(1)} \colon (\N^{\G} Y)^{\ast} = \B^{\ast} \times \N_{1}^{\ast} \times \N_{2}^{\ast} \to \N_{1}^{\ast}$, $\pi_{g}^{(2)} \colon (\N^{\G} Y)^{\ast} = \B^{\ast} \times \N_{1}^{\ast} \times \N_{2}^{\ast} \to \B^{\ast}$, $\pi_{g}^{(3)} \colon (\N^{\G} Y)^{\ast} = \B^{\ast} \times \N_{1}^{\ast} \times \N_{2}^{\ast} \to \N_2^{\ast}$. 

\begin{Def}
\hspace{0.5cm} \emph{i)} A \emph{boundary symbol} is an element $a \in S^{k,l}(\Lambda_b) \subset C^{\infty}(\Lambda_b)$ with $k,l \in \Rr$ such that
\[
|V_1 \cdots V_j a(\alpha)| \leq C(1 + |\alpha|)^{l- j_1} (1 + |\pi_b(\alpha)|)^{k-j_2}
\]

for all homogeneous vector fields $\{V_1, \cdots, V_l\}$ on $\Lambda_b$ of degree $\leq 1$. 
Here $j_1$ denotes the number of vector fields normal to $\B^{\ast}$ and $j_2$ denotes the number of vector fields lifted from $\B^{\ast}$.

\hspace{0.5cm} \emph{ii)} A \emph{singular Green symbol} is an element $a \in S^{m,l,k}(\Lambda_g) \subset C^{\infty}(\Lambda_g)$ with $m,l,k \in \Rr$ such that 
\[
|V_1 \cdots V_l a(\alpha)| \leq C(1 + |\pi_g^{(1)}(\alpha)|)^{m - j_1} (1 + |\pi_g^{(2)}(\alpha)|)^{m-j_2} (1 + |\pi_g^{(3)}(\alpha)|)^{m - j_3}
\]

for all homogeneous vector fields $\{V_1, \cdots, V_l\}$ on $\Lambda_g$ of degree $\leq 1$. 
Here $j_1$ denotes the number of transversals of the first kind, $j_2$ denotes the number of tangential vector fields lifted from $\B^{\ast}$
and $j_3$ denotes the number of transversals of the second kind. 
\label{Def:symbols}
\end{Def} 

\begin{Rem}
\emph{i)} There is also an invariant description of symbol spaces which consist of elements which are \emph{classical}, i.e. smooth functions
which admit asymptotic expansions. These symbols are particularly relevant in the study of the principal symbol of the operator calculus and subsequent study of Fredholm conditions for operators in the calculus. This invariant description is obtained with the help of radial compactifications of the above conormal bundles.
We refer to \cite{BS} for more details.

\emph{ii)} Since the codimension $\nu$ embedding $Y \hookrightarrow M$ is assumed to be admissible, by the tubular neighborhood theorem \cite{AIN} we can assume $M \cong Y \times \Rr^{\nu}$. 
Hence $\G$ locally takes the form $\G = \G_Y^Y \times \Rr^{\nu} \times \Rr^{\nu}$. Similarly, $Z \cong \G_Y^Y \times \Rr^{\nu}$ and $\Zop \cong \Rr^{\nu} \times \G_Y^Y$. 
\label{Rem:symbols}
\end{Rem}

\textbf{Kohn-Nirenberg quantization:} Let $\Phi \colon \G \to \N^{\G} \Delta_Y$ be the singular Green normal fibration associated to the embedding $\Delta_Y \hookrightarrow \G$. 
Then $U \in \U(\G)$ using Remark \ref{Rem:symbols}, \emph{ii)} is for a given symbol $u \in S^{m,l,k}(\Lambda_g)$ defined by
\[
(U f)(\gamma', \gamma_n) = \int_{\G_{s(\gamma)}} \int_{(\N^{\G} \Delta_Y)_{r_{\partial}(\gamma')}^{\ast}} e^{i \scal{\Phi(\gamma \eta^{-1})}{\xi}} u(r_{\partial}(\gamma'), \xi) f(\eta) \,d\xi\,d\mu_{s(\gamma)}(\eta)
\]

where $f \in C_c^{\infty}(\G), \ \gamma = (\gamma', \gamma_n) \in \G = \H \times \Rr^{\nu} \times \Rr^{\nu}$ and with phase function $\varphi(\gamma, \theta) = \scal{\Phi(\gamma)}{\theta}$. 
Let $\Psi \colon Z \to \N^{Z} \Delta_Y$ be the normal fibration of the embedding $\Delta_Y \hookrightarrow Z$.
Then $B \in \B(Z)$ is for a given symbol $b \in S^{k,l}(\Lambda_b)$ defined by 
\[
(B f)(z) = \int_{\G_{q(z)}} \int_{(\N^{Z} \Delta_Y)_{p(z)}^{\ast}} e^{i \scal{\Psi(z \gamma^{-1})}{\xi}} b(q(z), \xi) f(\gamma) \,d\xi \,d\mu_{q(z)}(\gamma)
\]
with phase function $\psi(z, \theta) = \scal{\Psi(z)}{\theta}$. 
Analogously, $C \in \C(\Zop)$ is defined as the adjoint operator of a boundary operator $B \in \B(Z)$ via 
\[
(Cf)(\gamma', \gamma_n) = \int_{Z_{s(\gamma)}} \int_{(\N^{\Zop} \Delta_Y)_{r_{\partial}(\gamma')}^{\ast}} e^{i\scal{\Psi^t(\gamma^{-1} z)}{\xi}} c(r_{\partial}(\gamma'), \xi) f(z) \,d\xi \, d \lambda_{s(\gamma)}^t(z)
\]
where $f \in C_c^{\infty}(Z), \ \gamma = (\gamma', \gamma_n) \in \G$. 
We denote by $\Psi^t \colon \Zop \to \N^{\Zop} \Delta_Y$ the corresponding normal fibration with phase function
given by $\psi^t(z, \theta) = \scal{\Psi^t(z)}{\theta}$. 

We define the algebra $\Phi(\G, \H)$ which depends on the generalized morphism $\H \dashrightarrow \G$. 

\begin{Def}
Let $\H \dashrightarrow \G$ be the generalized morphism of Lie groupoids which is implemented by the groupoid actions
on the fiberwise smooth spaces $Z \cong Z^t$. Fix the orders $(m_g, k_g, k_c, k_b) \in \Rr^4$ and the codimension $\nu := \codim(Y)$.
We set $l_g = -m_g - k_g - k, \ l_b = -k_b - \frac{\nu}{2}, \ m_c = -k_c - \frac{\nu}{2}$ and we let 
$k_g > 0, \ k_c > 0, \ k_b > 0, \ m_g < -\frac{\nu}{2}, \ m_g + k_g > -\frac{\nu}{2}$.  

Denote by $\A(m_g, k_g, k_c, k_b)$ the set of matrices of the form
\[
\begin{pmatrix} U & C \\
B & S \end{pmatrix}, \ U \in \U^{m_g, k_g, l_g}(\G), \ C \in \C^{m_c, k_c}(\Zop), \ B \in \B^{k_b, l_b}(Z), \ S \in \Psi(\H).
\]
whose entries belong to the operator classes
\begin{align*}
(\textrm{singular Green})\qquad	& \U^{m_g, k_g, l_g}(\G) := I_{cl}^{m_g, k_g, l_g}(\G, \Lambda_g) &\subset&\ I^{\ast}(\G, \Lambda_g), \\
(\textrm{boundary})\qquad		& \B^{k_b,l_b}(Z) := I_{cl}^{k_b, l_b}(Z, \Lambda_b) &\subset&\ I^{\ast}(Z, \Lambda_b), \\
(\textrm{co-boundary})\qquad		& \C^{m_c, k_c}(\Zop) := I_{cl}^{m_c, k_c}(\Zop, \Lambda_c) &\subset&\ I^{\ast}(\Zop, \Lambda_c), \\
(\textrm{pseudodifferential})\qquad		& \Psi(\H) := I_{cl}^{0}(\H, \Lambda_{\partial}) &\subset&\ I^{\ast}(\H, \Lambda_{\partial}),
\end{align*}
which are obtained from the corresponding symbol classes via the Kohn-Nirenberg quantization.
We finally define the algebra of \textit{Green operators} as
\[
\Phi(\G, \H) := \sum \A(m_g, k_g, k_c, k_b)
\]

where the sum is a non-direct sum over all tuples $(m_g, k_g, k_c, k_b)$.

\label{Def:operators}
\end{Def}

We omit the proof of the following result since it is rather technical and refer to \cite{BS} for details.
\begin{Thm}
Given the generalized morphism $\H \dashrightarrow \G$ associated to an admissible embedding of Lie manifolds $Y \hookrightarrow M$. Then $\Phi(\G, \H)$ is an associative $\ast$-algebra. 
\label{Thm:algebra}
\end{Thm}

One can prove more: The collection of operators of the type $\begin{pmatrix} P + U & C \\ B & S \end{pmatrix}$,
where $U, C, B, S$ are Green operators as defined in the previous Theorem and $P \in \Psi^{0}(\G)$ is a pseudodifferential operator on $\G$,
also forms an algebra.
This follows since by the rules of composition $\U(\G)$, the class of singular Green operators, is a $\Psi^{0}(\G)$-module and
thus $\{P + U : P \in \Psi^{0}(\G), \ U \in \U(\G)\}$ forms an algebra. 
For later reference we also record here the continuity properties of the operators in the calculus on the appropriate $L^2$-based Sobolev spaces.
\begin{Prop}
\emph{1)} A singular Green operator $U \in \U(\G)$ is a bounded linear operator $U \colon L^2(\G) \to L^2(\G)$. 

\emph{2)} A boundary operator $B \in \B(\G, \H)$ is a bounded linear operator $B \colon L^2(\G) \to L^2(Z)$.

\emph{3)} A coboundary operator $C \in \C(\G, \H)$ is a bounded linear operator $C \colon L^2(Z) \to L^2(\G)$. 

\label{Prop:L2}
\end{Prop}

These results can be extended to the case of operators which have order $> 0$ and which act on the $L^2$-based Sobolev spaces. 


\subsection{Non-commutative completion}

We define a comparison algebra for operators in the calculus associated to an admissible embedding of Lie manifolds. 
Then we show that the process of microlocalization of a given embedding has functorial properties. 
This makes it possible for us to define a representation of the previously introduced calculus, depending on a groupoid correspondence,
as operators acting on $L^2$-based Sobolev spaces. This representation is shown to be isomorphic to the comparison algebra. 
In the proof of the latter representation theorem we make use of the corresponding representation theorem proven in \cite{ALN} for the pseudifferential operators on Lie groupoids, but
otherwise our proof is very different. 

We first introduce operators that correspond to the restriction of a function on $M$ to $Y$ and its $L^2$-adjoint.

\begin{Def}
For a given embedding $Y \hookrightarrow M$ in $\EmbV$ we fix the following notion. We call two bounded operators 
$j^{\ast} \colon L_{\V}^2(M) \to L_{\W}^2(Y)$ and $j_{\ast} \colon L_{\W}^2(Y) \to L_{\V}^2(M)$ such that
$j_{\ast}$ is the $L^2$-adjoint of $j^{\ast}$ and $j^{\ast} j_{\ast} = \id_{L_{\W}^2(Y)}$ a \emph{generating (boundary, co-boundary) pair}.
\label{Def:genpair}
\end{Def}

We refer to the proof of the following Proposition for the construction of a generating pair. It is based on  \cite[Lem. 2]{AMMS}. 

\begin{Prop}
Let $j \colon Y \hookrightarrow M$ be an arrow in $\EmbV$. Then there is a generating pair $(j_{\ast}, j^{\ast})$ (canonically) associated to $j$.
\label{Prop:pair}
\end{Prop}

\begin{proof}
Let $j^{\ast}$ the restriction induced by pullback with $j$. By themselves these operators do not yet yield a generating pair. We modify them by use of a standard \textit{order reduction technique} as follows:\\
By \cite[Theorem 4.7]{AIN} this yields a continuous map
\[
j^{\ast} \colon H_{\V}^s(M) \to H_{\W}^{s - \frac{\nu}{2}}(Y), \qquad s>\frac{\nu}{2}
\]
where $H_{\V}^s(M)$ and $H_{\W}^{s'}(Y)$ are the $L^2$-based Sobolev spaces on $M$ and $Y$ respectively, defined by completion of $C_c^\infty$-functions.

We assume without loss of generality that the codimension of the embedding $j(Y)\subset M$ is constant $\nu = 1$ (otherwise we simply have to adapt the order reductions below accordingly).
Consider the bounded and invertible operator given by
\[
\begin{pmatrix} \Delta \\ j^{\ast} \end{pmatrix} \colon H_{\V}^2(M) \longrightarrow \begin{matrix} L_{\V}^2(M) \\ \oplus \\ H_{\W}^{\frac{3}{2}}(Y) \end{matrix}. 
\]

Here $\Delta$ is the Laplace operator associated to a fixed compatible metric $g$ on $M$. For the definition of the following order reductions in the Lie calculus we refer to \cite[Section 8]{B2}. 
Denote by $\lambda_{\partial}^{\frac{3}{2}}:H^s_\W(Y)\rightarrow H^{s-\frac{3}{2}}_\W(Y)$ an order reduction isomorphism of order $\frac{3}{2}$ on $Y$ and by $\lambda^{-2}$ an order reduction of order $-2$ on $M$. 
We obtain an isomorphism
\[
\begin{pmatrix} \Delta \lambda^{-2} \\ \lambda_{\partial}^{\frac{3}{2}} j^{\ast} \lambda^{-2} \end{pmatrix} \colon L_{\V}^2(M) \stackrel{\lambda^{-2}}{\longrightarrow} H^2_\V(M) \  \stackrel{\begin{pmatrix} \Delta \\ j^{\ast} \end{pmatrix}}{\longrightarrow}\ \begin{matrix} L_{\V}^2(M) \\ \oplus \\ H^{\frac{3}{2}}_{\W}(Y) \end{matrix} \ \stackrel{\id\times \lambda_\partial^{\frac{3}{2}}}{\longrightarrow} \ \begin{matrix} L_{\V}^2(M) \\ \oplus \\ L_{\W}^2(Y) \end{matrix}. 
\]

In particular the operator $B = \lambda_{\partial}^{\frac{3}{2}} j^{\ast} \lambda^{-2} \colon L_{\V}^2(M) \to L_{\W}^2(Y)$ has a right inverse which we denote by $C$. 
We check that the operator $C^{\ast} C$ is strictly positive: 
\[
\|v\|_{L_{\W}^2} = \|B C v\|_{L_{\W}^2} \leq \|B\|_{\L(L_{\V}^2, L_{\W}^2)} \|Cv\|_{L_{\V}^2}
\]

hence $\|Cv\| \geq c \|v\|$ for some $c > 0$. 
We set $S := (C^{\ast} C)^{-\frac{1}{2}}$ which is a $0$-order $\W$-pseudodifferential operator on $Y$. 
By an abuse of notation we now use the same symbol $j^{\ast}$ to denote the operator $S C^{\ast}$ and denote by $j_{\ast}$ the operator $CS$ which furnishes the desired generating pair. 
\end{proof}

\begin{Rem}
The proof shows in which sense $j^\ast$ is a microlocalization of the pullback by $j$. Taking the pullback results in a loss of Sobolev regularity, which can be formulated on a microlocal level by use of Sobolev wave front sets \cite{HHyp}. The order reductions, which are elliptic pseudo-differential operators, do not move singularities, and are used to restore this regularity. As such, $j^\ast$ is an operator with the same microlocally positioned singularities as the pullback by $j$, but of a different Sobolev strength.
\end{Rem}

\begin{Def}
For each $j \colon (Y, \W) \hookrightarrow (M, \V)$ in $\EmbV$, $\Phi_{\V}(j)$ is the $\ast$-closed subalgebra of $\L(L_{\V}^2(M) \oplus L_{\W}^2(Y))$ generated by the set
\[
\left\{\begin{pmatrix} j_{\ast} j^{\ast} & j^{\ast} \\ j_{\ast} & S \end{pmatrix} : S \in \Psi_{\W}(Y), \ (j^{\ast}, j_{\ast}) \ \text{generating pair}\right\}. 
\]
\label{Def:comparison}
\end{Def}



\textbf{Functoriality:} We show how the algebra $\Phi(\G, \H)$ associated to the embedding $j \colon (Y, \B, \W) \hookrightarrow (M, \A, \V)$, which we have constructed via the Kohn-Nirenberg quantization, can be viewed as a covariant functor $\Phi \colon \EmbV \to C_b^{\ast}$.
At first there are two obvious natural transformations $\EmbV \to \LG_b$ and $\EmbV \to \LA_b$ associating to the embedding $j$ the 
generalized morphism of integrating groupoids $\H \dashrightarrow \G$ and of algebroids $[\B^{\ast} \longleftarrow T^{\ast} Z \longrightarrow \A_{|Y}^{\ast}]$.
We denote these natural transformations by $\lift_{qu}$ and $\lift_{cl}$ which stands for \emph{quantized lift} and \emph{classical lift} respectively. 
As usual in quantization we associate operators to the classical side. In our case this is accomplished via the Kohn-Nirenberg quantization procedure. 
There is an intermediate step in the construction of these operators using Fourier integral operator theory. 
Here we also associate to the generalized morphism $\H \dashrightarrow \G$ a generalized morphism $T^{\ast} \H \dashrightarrow T^{\ast} \G$ in $\SG_b$ of so-called Coste-Dazord-Weinstein (CDW) symplectic groupoids, cf. \cite{LV} and \cite{BS}.  
Altogether, consider the diagram 
\begin{figure}[h]
\begin{tikzcd}
\EmbV\drar[bend right, "\lift_{cl}", swap]\rar[bend left, "\lift_{qu}"] 
 & \LA_b \rar["\mathrm{KN}"]& \C_b^{\ast} \\
 & \LG_b \uar["\A^{\ast}"] \rar["\mathrm{CDW}"] & \SG_b \uar["\mathrm{FIO}", swap]
\end{tikzcd}
\end{figure}

Here $\mathrm{KN}$ stands for the Kohn-Nirenberg quantization, $\mathrm{FIO}$ is the functor which assigns Fourier integral operators to a 
generalized morphism of symplectic (CDW) groupoids and $\mathrm{CDW}$ is the canonical functor from Lie groupoids to CDW-groupoids. 
Altogether we obtain a covariant functor $\Phi \colon \EmbV \to \C_b^{\ast}$. We will compare this with the previously introduced functor $\Phi_{\V}$. 



\textbf{Representation:} Define the $\ast$-representations $\varrho \colon \U(\G) \to \End(C^{\infty}(M))$ of the singular Green operators via the identity
$(\varrho(G) \circ r)(f) = G(f \circ r)$ for each $f \in C^{\infty}(M)$. Also define the $\ast$-representation $\varrho_{\partial} \colon \Psi(\H) \to \End(C^{\infty}(Y))$
via $(\varrho(S) \circ r_{\partial})(g) = S(g \circ r_{\partial})$ for each $g \in C^{\infty}(Y)$. Here $r \colon \G \to M$ denotes the range
map of the groupoid $\G$ and $r_{\partial} \colon \H \to Y$ the range map of the groupoid $\H$. The $\ast$-homomorphisms $\varrho, \ \varrho_{\partial}$ are also called \emph{vector representations}, cf. \cite{ALN}. 
Using these vector representations we can define also a representation of our algebra which is given as a natural transformation.

\begin{Thm}
\emph{i)} The assignment $\EmbV \ni j \mapsto \Phi(j) \in C_b^{\ast}$ furnishes a covariant functor
$\Phi \colon \EmbV \to C_b^{\ast}$. 

\emph{ii)} There is an essentially surjective natural transformation $\varrho_{\Phi} \colon \Phi \to \Phi_{\V}$
such that $\varrho_{\Phi|\G} = \varrho$ and $\varrho_{\Phi|\H} = \varrho_{\partial}$. 

\label{Thm:repr}
\end{Thm}

\begin{proof}
\emph{i)} We have to study the functor $\Phi \colon \Emb_{\V} \to C_b^{\ast}$. 
The interesting point is the map on morphisms. Hence we have to show that for any given admissible embedding
$Y \hookrightarrow M$ the algebra $\Phi(\G, \H)$ implements a generalized morphism of $C^{\ast}$-algebras.
Fix an admissible embedding $Y \hookrightarrow M$, i.e. a morphism in $\EmbV$. 

Consider the $L^2$-completed order-$0$ algebra $\Phi(\G, \H)$ which can be written in terms of the matrix $\begin{pmatrix} \U & \C \\ \B & \Psi_{\partial} \end{pmatrix}$.
Here $\U$ denotes the $L^2$-completion of the algebra of singular Green operators and $\Psi(\H)$ denotes the $L^2$-completion
of the algebra of pseudodifferential operators on $\H$.

By the rules of composition for $\Phi(\G, \H)$ we have $\U \cdot \C \subset \C$ and $\Psi(\H) \cdot \B \subset \B$.
We can check that $\C$ is a right-Hilbert $\U$-module and $\B$ is a right Hilbert $\Psi(\H)$-module.
Additionally, we need to show that there is a $\ast$-homomorphism $\varphi \colon \U \to \L_{\Psi(\H)}(\B)$ taking
values in the adjointable operators. There is a scalar product $_{\Psi_{\partial}}\scal{\cdot}{\cdot} \colon \B \times \B \to \Psi_{\partial}$ such that
$_{\Psi(\H)}\scal{\varphi}{\psi} = _{\Psi(\H)}\scal{\psi}{\varphi}$ and $_{\Psi(\H)}\scal{\varphi}{\varphi} \geq 0$. 
Additionally $\|\varphi\|^2 = \|_{\Psi(\H)}\scal{\varphi}{\varphi}\|$ defines a norm with regard to which $\B$ is complete.
In our case we define $_{\Psi_{\partial}}\scal{B_1}{B_2} = B_1 \circ B_2^{\ast} \in \Psi_{\partial}$. 
Also $\varphi \colon \U \to \L_{\Psi_{\partial}}(\B)$ is given by $\varphi(G) \colon \B \to \B$ where $G$ is a singular Green operator.
The latter is defined as $\varphi(G)(B) = (G \cdot B^{\ast})^{\ast} = B \cdot G^{\ast}$. 
Then check that $_{\Psi(\H)}\scal{B_1}{\varphi(G) B_2} = _{\Psi(\H)}\scal{\varphi(G)^{\ast} B_1}{B_2}$ which holds
since by definition
\begin{align*}
_{\Psi(\H)}\scal{B_1}{\varphi(G) B_2} &= B_1 (\varphi(G) B_2)^{\ast} = B_1 (B_2 G^{\ast})^{\ast} \\
&= B_1 G B_2^{\ast} = (\varphi(G)^{\ast} B_1) B_2^{\ast} \\
&= _{\Psi(\H)}\scal{\varphi(G)^{\ast} B_1}{B_2}. 
\end{align*}

Secondly, $\phi$ is a homomorphism because $\varphi(G_1 G_2) B = B G_2^{\ast} G_1^{\ast} = \varphi(G_1) \varphi(G_2) B$.
Also we have that
\[
_{\Psi(\H)}\scal{\varphi(G)  B_1}{B_2} = _{\Psi(\H)}\scal{B_1}{\varphi(G^{\ast}) B_2}.
\]

Note that the left hand side equals $(B_1 G^{\ast}) B_2^{\ast}$ while the right hand side equals $B_1 (B_2 G)^{\ast} = B_1 G^{\ast} B_2^{\ast}$.
Hence $\varphi$ is a well-defined $\ast$-homomorphism.
This yields the desired generalized morphism in $C^{\ast}$ and we have shown that $\Phi$ is a covariant functor.

\emph{ii)} We fix the embedding functor \ $\widehat{}_b \ \colon C^{\ast} \hookrightarrow C_b^{\ast}$ from Proposition \ref{Prop:inclCstar}.
Let $j \in \Mor(\E_{\V})$ be the admissible embedding $j \colon Y \hookrightarrow M$ and denote by $\W := \{V_{|Y} : V \in \V, \ V_{|Y} \ \text{tangent to} \ Y\}$
the induced Lie structure of $Y$.  
Also fix the $\ast$-homomorphisms $\varrho \colon \U(\G) \to \U_{\V}(M, Y)$ and $\varrho_{\partial} \colon \Psi(\H) \to \Psi_{\W}(Y)$ as defined above.
We then define the natural representation $\varrho_{\Phi} \colon \Phi \to \Phi_{\V}$ and check that it is surjective natural transformation.
The following diagram of generalized morphisms in $C^{\ast}$ commutes
\[
\xymatrix{
\Psi(\H) \ar@{-->}[d]_{\varrho_{\Phi|\H}} \ar@{-->}[r]^{\Phi(j)} & \U(\G) \ar@{-->}[d]_{\varrho_{\Phi|\G}} \\
\Psi_{\W}^{\ast}(Y) \ar@{-->}[r]^{\Phi_{\V}(j)} & \U_{\V}(M, Y).
}
\]

Where we define $\varrho_{\Phi|\G} := \widehat{}_b \ \circ \varrho$ and $\varrho_{\Phi|\H} := \widehat{}_b \ \circ \varrho_{\partial}$. 
In particular the surjective $\ast$-homomorphism $\varrho_{\partial} \colon \Psi(\H) \to \Psi_{\W}(Y)$ yields
a $\Psi_{\W}^{\ast}(Y) - \Psi(\H)$ bimodule.
Also the surjective $\ast$-homomorphism $\varrho \colon \U(\G) \to \U_{\V}(M)$ yields a $\U_{\V}(M) - \U(\G)$ bimodule.
By definition $\Phi_{\V}(j)$ is a $\U_{\V} - \Psi_{\W}$ bimodule and $\Phi(j)$ is a $\U - \Phi(\H)$ bimodule.
Then 
\begin{align*}
\Phi_{\V}(j) \circ \varrho_{\Phi|\H} &= \Phi_{\V}(j) \hat{\otimes}_{\Psi_{\W}} \widehat{\varrho_{\partial}}, \\
\varrho_{\Phi|\G} \circ \Phi(j) &= \widehat{\varrho} \hat{\otimes}_{\U(\G)} \Phi(j). 
\end{align*}

Here $\hat{\otimes}$ denotes the Rieffel tensor product, cf. section \ref{Cstarb} and \cite{L}. 
The surjectivity of $\varrho_{\Phi}$ follows from the surjectivity of the strict morphisms $\varrho_{\partial}, \ \varrho$. 
\end{proof}

\section*{Acknowledgements}

For useful discussions I thank Magnus Goffeng, Victor Nistor, Elmar Schrohe and Georges Skandalis.
This work was supported through the programme \emph{Oberwolfach Leibniz Fellows} by Mathematisches Forschungsinstitut Oberwolfach in 2015.

\end{document}